\documentclass[11pt]{amsart}

\usepackage{amsmath,amsthm, amssymb, color}
\usepackage{graphicx}
\usepackage{fullpage}
\usepackage[normalem]{ulem}

\usepackage{amsmath, amsthm, amssymb, bm, mathtools}
\usepackage{enumerate}

\newcommand{\R}{{\mathbb {R}}}

\newcommand{\Z}{{\mathbb Z}}

\newtheorem{Theorem}{Theorem}

\newtheorem {lemma} [Theorem]    {Lemma}
\newtheorem {corollary}  [Theorem]    {Corollary}

\newtheorem {proposition}[Theorem]    {Proposition}
\newtheorem {theorem}[Theorem]    {Theorem}
\newtheorem{remark}[Theorem]{Remark}

\newcommand {\eps}{\varepsilon}

\title[A switching identity for cable-graph loop-soups]{A switching identity for cable-graph loop soups and Gaussian free fields}
\author{Wendelin Werner}
\address{University of Cambridge}

\begin{document}

\begin {abstract}
We derive a ``switching identity'' that can  be stated for critical Brownian loop-soups or for the Gaussian free field on a cable graph: It basically says that at the level of cluster configurations and at the more general level of the occupation time fields,  conditioning two points on the cable-graph to belong to the same cluster of Brownian loops (or equivalently to the same sign-cluster of the GFF) amounts to adding a random odd number of independent Brownian excursions between these points to an otherwise unconditioned configuration. Such explicit simple descriptions of the conditional law of the clusters when a connection occurs have various direct consequences, in particular about the large scale behaviour of these sign-clusters on infinite graphs.
\end {abstract}

\maketitle 
\section {Introduction}

\subsection* {Sketchy overview}
The main purpose of this paper is to derive a new result for a particular percolation model (critical percolation of Brownian loops on cable-graphs), which is basically a collection of random non-interacting Brownian loops on the graph (with infinitely many small loops in any portion of the graph) that is of some significance in Physics (as we shall see, the result can be also formulated in terms of the Gaussian free field). We will in particular see that if one conditions two given points $x$ and $y$ to be in the same percolation cluster, then the conditional law of this cluster can be described in terms of the cluster containing $x$ (and $y$) of the overlay of an unconditioned percolation of loops configuration with an odd number $N$ of Brownian excursions joining $x$ and $y$. The number $N$ will typically be equal to one when $x$ and $y$ are very far apart (or exactly equal to one when the previous conditioning was on $x$ and $y$ to be on the boundary of the same cluster).  So, in a nutshell: In this web of independent Brownian loops, conditioning $x$ and $y$ to be connected amounts exactly to adding an odd number of independent paths joining these two points to the unconditioned picture (see the sketch below).

\begin{figure*}[h]
  \centering
  \includegraphics[width=.9\textwidth]{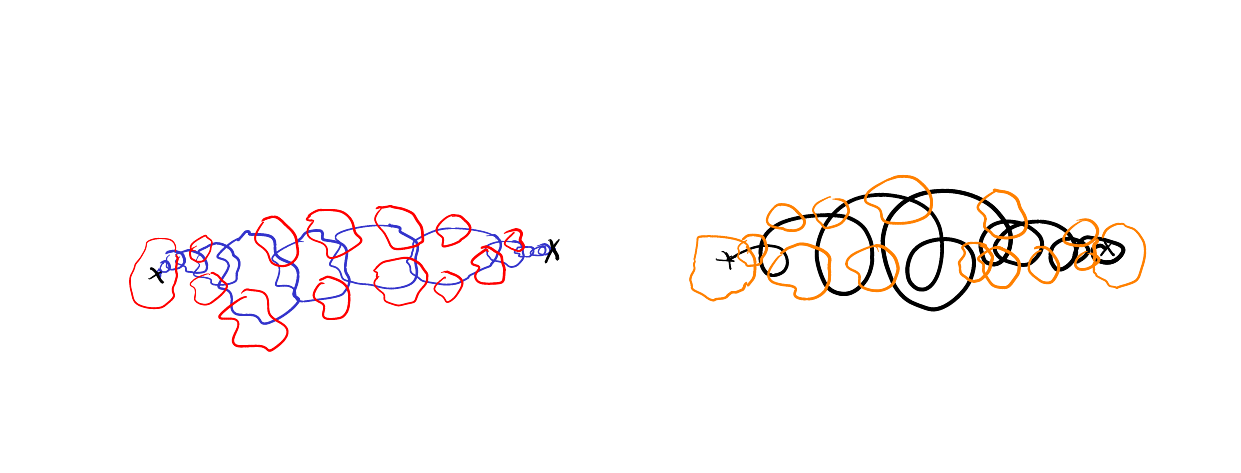}
  \caption{\label{fig} The two sketched constructions (conditioning on connecting $x$ and $y$ vs. just adding one excursion between $x$ and $y$) give the same measure on clusters.}
  \end{figure*}

  This result may at first sight seem quite surprising and counterintuitive. Indeed, connections in loop-soups are created by chains of Brownian loops so that, when a connection occurs between $x$ and $y$, one would rather think of two (or more generally an even number of) connecting paths between these two points. On the other hand, adding an odd number of paths to a loop-soup would create an odd total number of such connecting paths. Some version of this first fact (i.e., that one has an even number of paths) will actually show up in the proof of the general theorem.  Our result however says that the cluster can nevertheless equivalently be viewed as an odd number of excursions joining $x$ to $y$ with a bunch of independent Brownian loops agglomerated to it. The reader is encouraged to scribble little pictures to see that such a bijection is far from obvious and cannot be obtained by finitely many simple ``rewirings/recombinations'' between loops and excursions.

Besides of their own independent interest (this is related to the general question of how ``bosonic fields transmit forces''), our switching results have consequences in a number of directions.
One example is that when fixing $x$ and sending the point $y$ to infinity (in a graph like $\Z^d$ for any $d \ge 3$, but this holds also on most infinite graphs), one gets a direct construction and a strikingly simple description of the ``incipient infinite measure'' for this loop-percolation model (i.e., the law of the percolation ``conditioned on $x$ being in an infinite cluster'') as the overlay of one unconditioned loop percolation configuration with one Brownian excursion from $x$ to infinity.

\subsection*{Background}
Before properly stating our main results, let us spend two pages to review some basic and/or relevant facts about the Gaussian free field and the (critical) Brownian loop soups on  cable graphs, following results of Le Jan, Lupu and others (for background, we refer to \cite {WP}). The reader acquainted with those can safely jump to the {\em Switching Property} section.

Through most of this paper, we consider a connected cable-graph $G$ (with a possibly non-empty boundary) on which the Green's function is finite (this condition is in fact not really needed, but it will make the exposition simpler -- see Remark \ref {compact}). This includes cable-graphs such as the union of all the edges of $\Z^d$ (viewed as segments in $\R^d$) for $d \ge 3$ or any connected proper subset of the union of the edges of $\Z^d$ for $d \ge 1$.

Our graphs will also sometimes implicitly contain boundary points (corresponding to the fact that we take the Green's function of the Laplacian with Dirichlet boundary condition at those points -- i.e., the Brownian motion on the cable graph will be killed at those points). This will in particular be the case when we will mention loop-soups in $G \setminus \{ x \}$ (or $G \setminus \{ x,y \}$), where $\{ x \}$ (resp. $\{x,y\}$) will then be part of the boundary.

\begin {itemize}
\item (Cable-graph GFF).
The cable-graph Gaussian free field (GFF) is a random function $(\Gamma (x), x \in G)$ that can be viewed as the natural generalization of Brownian motion when one replaces the time-axis by $G$ (and one conditions it to be equal to $0$ on the points of $\partial G$, and adds some conditions at infinity if the graph is infinite). One simple way to turn this intuition into a rigorous definition is to say that $(\Gamma (x), x \in G)$ is a centered real-valued Gaussian process defined on $G$ with covariance function given by the Green's function on the cable-graph. When one restricts this cable-graph GFF to the discrete set of vertices of the graph, then one obtains an instance of the discrete GFF, which has been a basic building block of Quantum Field Theory for many decades (the discrete GFF is sometimes referred to as the bosonic free field in the Physics literature).

Looking at the cable-graph GFF allows to use Brownian-type features such as the reflection principle that are not so well-suited to discrete graphs. In this context, it is extremely natural to define
the zero-set $Z(\Gamma)= \{ x \in G, \ \Gamma (x) = 0 \}$ of $\Gamma$.
A {\em sign-cluster of $\Gamma$} is then defined to be a connected component of $G \setminus Z$.  In other words, two points $x$ and $y$ on the cable graph are in the same sign-cluster  if there exists a continuous path $\gamma : [0, 1] \to G$ with $\gamma (0) = x$, $\gamma (1) = y$ and $\Gamma (\gamma(s)) \not= 0 $ for all $s \in [0,1]$. We will denote this event by $x \leftrightarrow y$. Since sign-clusters are determined by $Z(\Gamma)$, they can also be viewed as functions of the square of $\Gamma$, or as functions of $|\Gamma|$. The aforementioned reflection principle states essentially that one can resample independently the sign of the GFF on each of the sign-clusters without changing its law. This in particular implies immediately that
$E[ \Gamma(x) \Gamma(y)] = E [ \Gamma(x) \Gamma(y) 1_{x \leftrightarrow y }]$ (because when $x \not\leftrightarrow y$, then the signs of $\Gamma(x)$ and $\Gamma(y)$ are chosen independently). This line of thought led Lupu \cite {L1} to derive a simple expression for $P [ x \leftrightarrow y \mid  \Gamma^2 (x)=a^2, \Gamma^2(y) =b^2 ]$ in terms of the densities of $(\Gamma(x), \Gamma(y))$ at $(a,b)$ and $(a, -b)$, and then a simple expression for $P[x \leftrightarrow y]$ by integrating over the law of $(\Gamma(x), \Gamma(y))$. In particular, in $\Z^d$ for $d \ge 3$, this probability decays like a constant times $1/\|x-y\|^{d-2}$ as $\|x-y\| \to \infty$. This formula has been the basis of a number of subsequent work on the sign-clusters.

\item (Brownian loop-soup and Brownian excursions).
One can define Brownian motion on the cable-graph -- which is not to be confused with the GFF (the latter is a function from $G$ to $\R$ while the former is a function from a subset of $\R$ into $G$). This is a time-parameterized process that moves on the cable-graph, just like one-dimensional Brownian motion does when it is in some edge (i.e., not at a vertex where several edges meet), and when it is at a vertex $x$ where several edges meet, then (and this definition can be easily made rigorous using basic excursion theory for Brownian motion), it chooses one of its adjacent edges uniformly at random to start each excursion away from $x$.
One can then define the natural variants of Brownian motion, namely the Brownian excursion measures and the Brownian loop measures.
For any two points $x$ and $y$, the  ``Brownian excursion measure'' $\nu_{x, y, G}$ from $x$ to $y$ in $G$ is then a measure  on the set of paths $\beta: [0, \tau_\beta] \to G$ with $\beta(0) = x$, $\beta (\tau) = y$ and $\beta (0,\tau) \subset G \setminus \{ x, y\}$. When $x=y$, this is an infinite measure (with an infinite mass on the set of small excursions), while when $x \not= y$, it is a finite measure (that can therefore be renormalized to be a probability measure). All these measures are naturally related to basic features of the cable-graph $G$ viewed as an electric network.

The Brownian loop-measure in the cable-graph is an infinite measure on loops (with infinite mass put on very small loops). This measure is then used to define our main player here, namely the Brownian loop-soup (in the terminology of \cite {LW}), which is a Poisson point process of Brownian loops in the cable-graph with intensity given by the Brownian loop measure  [throughout the present paper, all these loop-soups will be the ones with this particular intensity, i.e., the multiplicative factor in front of this intensity will be fixed  -- the loop-soups with this intensity are sometimes referred to as the {\em critical} ones]. Brownian loop-measures and loop-soups were introduced and used in the continuum setting \cite {LW,SW}, then in the discrete case \cite {LT,LJ1} and in cable graphs \cite {L1}.
\end {itemize}

The Brownian loop-soup on the cable graph turns out to be directly related to the GFF in the following way: If for each $x \in G$, one adds up the local time at $x$ of all the loops in the loop-soup, then one obtains the {\em occupation time field} $(\Lambda (x), x \in G)$ of the loop-soup. It turns out (as observed by Le~Jan \cite {LJ1,LJ2} -- see \cite {LJ3} for an extensive overview and more references) that $(\Lambda (x), x \in G)$ is distributed exactly like $(\Gamma^2(x), x \in G)$. Furthermore \cite {L1}, the sign-clusters of $\Lambda$ are then exactly the connected components of the union of all the loops in the loop-soup and the zero-set of $\Lambda$ is exactly the set of exceptional points that are visited by no loop. In this way, sign-clusters of the GFF are just ``clusters of Brownian loops'' and the aforementioned formula by Lupu for $P[ x \leftrightarrow y]$ is the two-point function for percolation of Brownian loops.

One can therefore construct a GFF out of a Brownian loop-soup by first defining its square to be equal to $\Lambda$, and by then choosing a sign at random independently for each of cluster of Brownian loops. Through most of the present paper, we will implicitly use this construction/coupling, i.e., we will always have $\Lambda  = \Gamma^2$.

Many of the striking features of the GFF, such as its spatial Markov property or Dynkin's isomorphism can be revisited and reinterpreted using the Brownian loop-soup and its own striking combinatorial properties, in particular its rewiring properties studied in \cite {W1}. It is worth noting that historically, many results had been derived along the lines of Dynkin's isomorphism (see e.g. the monograph \cite {MR} and the references therein) before mathematicians realized the direct connection to loop-soups. Let us now very briefly describe how to interpret and derive Dynkin's isomorphism using this GFF/loop-soup correspondence, since related ideas will appear in this paper: Suppose that one conditions the value of $\Lambda (x)$ to be equal to $a^2$. On the GFF side, this means that one conditions $\Gamma(x)$ to be equal to $\pm a$, and the Markov property of the GFF then says that the conditional law of $\Gamma^2$ is that of $ ( \Gamma_0 + \Phi)^2$, where $\Gamma_0$ is a GFF conditioned to be $0$ at $x$, while $\Phi$ is the harmonic function in $G \setminus \{x\}$ with value $a$ at $x$ (and appropriate behaviour at infinity). On the loop-soup side, this conditioning only affects the loops that go through $x$. The union of those loops then become a Poisson point process ${\mathcal E}$ of excursions away from $x$ with intensity $a^2 \nu_{x,x}$ while the remaining collection of loops is still a loop-soup in $G \setminus \{ x \}$ (with occupation times distributed like $\Gamma_0^2$). Dynkin's isomorphism is the statement that these two descriptions of the conditional law of $\Lambda=\Gamma^2$ are the same, i.e. that the law of $ ( \Gamma_0 + \Phi)^2$ is the same as the law of the sum of $\Gamma_0^2$ with the occupation time field of ${\mathcal E}$.

Throughout the paper (and we have already done so a couple of times), conditioning on probability zero events (such as that the value at $x$ of $\Gamma$ is equal to some value, or that $x$ is on the boundary of a cluster) is easily made sense of because of the relation of the cable-graph GFF with one-dimension Brownian motion (or using the Gaussian framework).

The paper \cite {W2} outlined some results and conjectures about the geometry of large sign clusters in $\Z^d$, and how these features depend on the spatial dimension $d$.
In this direction, it is worth mentioning that intuitively speaking, when the dimension $d$ is large, the longer loops in a Brownian loop-soup become much less frequent, so that one expects the picture to resemble more that of ordinary critical percolation in high dimensions, whereas when $d \le 5$, the large loops will play an important role in the appearance of large clusters, so that one is in a very different universality class than critical Bernoulli percolation. Indeed, when $d=2$, it has been shown that the loop-soup clusters (and sign-clusters) can be described via CLE$_4$ loop ensembles (and their variants) in the scaling limit \cite {L2,ALS}.
There has been a fairly intense recent activity in the study of the large-scale properties (and exponents) of the sign clusters when $d \ge 3$ (see for instance \cite {LW3,DPR1,P,DW} -- and in 2024 alone, it has been the main focus of \cite {CD1,CD2,CD3,GN,DPR1,DPR2,GJ}).

\subsection*{The switching properties}

We are now ready to state our main results.
We start with a consequence of our more general theorem that can be formulated in terms of loop-soup clusters only (i.e., without reference to occupation time fields):
Let us consider a loop-soup in $G$ and choose two given different points $x$ and $y$ in $G$.
Note that the Poissonian nature of the loop-soup shows that conditioning two points $x$ and $y$ not to be visited by any loop just amounts to erasing all the loops that go through $x$ or $y$ (i.e., to consider a loop-soup ${\mathcal L}$ in $G \setminus \{ x,y\}$). It is then  also not difficult to make sense of the law of this loop-soup that is further conditioned on these points $x$ and $y$ being on the boundary of the same loop-soup cluster. This is the event that there exists a bi-infinite chain of loops $(\gamma_{i(k)})_{k \in \Z}$ in the loop-soup ${\mathcal L}$ with $\gamma_{i ( k-1)} \cap \gamma_{i(k)} \not= \emptyset$ for all $k$ and such that  $d( \gamma_{i(k)}, x) \to 0$ and $d(\gamma_{i(k)}, y) \to 0$ as $k \to -\infty$ and $k \to + \infty$ respectively [the way to make sense of this conditioning by this event of zero probability is similar to the conditioning of Brownian motion started from $0$ to be strictly positive on the time-interval $(0,1]$, giving rise to variants of the three-dimensional Bessel processes].
Then:

\begin {theorem}[The special case of cluster boundary points]
\label {theoremloopsperco}
\label {thm1}
Let ${\mathcal L}$ denote a Brownian loop-soup conditioned on $x$ and $y$ not being visited by any loop (i.e., ${\mathcal L}$ is a standard loop-soup in $G \setminus \{ x,y\}$). The law of the following two collections of clusters are identical:
\begin {itemize}
 \item The clusters of
 ${\mathcal L}$ when this loop-soup is conditioned on $x$ and $y$ being on the boundary of the same cluster.
 \item The clusters of the union of ${\mathcal L}$ with one independent Brownian excursion from $x$ to $y$ (with its endpoints $x$ and $y$ removed).
 \end {itemize}
This identity also holds when one considers the occupation times on top of the clusters.
\end {theorem}
Note that in both descriptions, the collection of loops that are not part of the cluster $C$ that $x$ and $y$ are on the boundary of is clearly just a loop-soup in $G \setminus C$ (if one conditions on $C$), so that the core of the theorem is the fact that the law of this cluster $C$ is the same in both cases (as sketched in Figure \ref {fig}). Removing the endpoints of the excursion in the second description is just to ensure that in both descriptions the points $x$ and $y$ are on the boundary of a cluster but not in the cluster itself.

\medbreak

This theorem can be viewed as a limiting case of the following  general version of the switching property. This is now a description of the conditional law of $\Gamma^2$ (i.e., of the loop-soup occupation times) given that  $x \leftrightarrow y$, $\Lambda(x)= \Gamma^2 (x)=a^2>0$ and $\Lambda(y) = \Gamma^2 (y)=b^2>0$. It now involves a Poisson number of independent excursions joining $x$ to $y$  {\em conditioned  to be odd} (note that if it  is odd, then there is at least one excursion):
\begin {theorem}[The general switching property]
\label {mainthm}
\label {thm2}
The conditional distribution of the critical Brownian loop-soup occupation time $\Lambda$ (i.e., of $\Gamma^2$) conditionally on $x \leftrightarrow y$, $ \Lambda (x)=a^2 > 0$ and $\Lambda (y)= b^2 > 0$, is the law of the sum of the occupation times of the following {\bf independent} inputs:
\begin {itemize}
 \item A critical Brownian loop-soup in  $G \setminus \{x,y\}$.
 \item A Poisson point process of excursions away from $x$ in $G \setminus \{y\}$ with intensity $a^2 \nu_{x , x, G \setminus \{y\}}$.
 \item A Poisson point process of excursions away from $y$ in $G \setminus \{x\}$ with intensity $b^2 \nu_{y , y, G \setminus \{ x \}}$.
 \item A Poisson point process of excursions joining $x$ and $y$ in $G$ with intensity $ab \nu_{x , y, G}$
 {\bf where the number of these excursions is conditioned to be odd.}
\end {itemize}
\end {theorem}
Let us stress the two parts of the theorem written in bold, i.e., that these four inputs are {\em independent} and that the number of excursions joining $x$ to $y$ is {\em conditioned to be odd} (which implies that a connection from $x$ to $y$ indeed holds, regardless of the other three inputs). Recall that the conditional probability that $x \leftrightarrow y$ given $\Lambda (x)$ and $\Lambda (y)$ is known (and also expressed in terms of the total mass of $\nu_{x , y, G}$) so that this theorem yields also description of the conditional law of $\Lambda 1_{x \leftrightarrow y}$ given $\Lambda(x)$ and $\Lambda(y)$, and in turn (since the law of the Gaussian vector $(\Gamma(x), \Gamma(y))$ is of course also known) a description of the unconditional law of $\Lambda 1_{x \leftrightarrow y}$. In this integrated version, the description is still simple and useful, but the four items loose their full independence (as the last three ones will be related via the now random values of $\Lambda(x)$ and $\Lambda (y)$).

\medbreak
To our knowledge, this is essentially the only case (except for tree-like graphs) where a percolation measure conditioned on the existence of a connection has such a simple description.
It is a rather powerful and useful tool: It allows to revisit/simplify the proofs of some of the existing results on cable-graph GFFs and their geometry, and gives access to further new results. Some of its immediate consequences that are given in Section \ref {S2} deal with the case where one of the points is sent to infinity. Let us state just one of these in this introduction namely  the existence and a simple description of the incipient infinite cluster (IIC) for this percolation of loops model and leave further comments and results in Sections \ref {S2} and \ref {S6}.    We consider here $G$ to be $\Z^d$ for $d \ge 3$ but as we will discuss in Section \ref {SLiouville}, this result is in fact valid for a wide class of infinite transient graphs.   We will use here the probability measure on ``excursions from the origin to infinity'' (which can be obtained by taking an unconditioned Brownian motion starting from $x$, and keeping only its part after its last visit of the origin):

\begin {theorem}[The IIC measure for  $\Z^d$ for all $d \ge 3$]
\label {IIC}
\label {thmIIC}
 The limit when $x \to \infty$ of the law of the square of the cable-graph GFF conditioned on $0 \leftrightarrow x$ exists.
 The law of this incipient infinite cluster measure, is that of the occupation time of the overlay of a Brownian loop-soup in $\Z^d$ with distribution reweighed by the square root of its total local time at the origin, with one independent Brownian excursion from $0$ to infinity.
\end {theorem}
Here, the convergence means that for any $K$, the law of the occupation times viewed as random continuous functions defined on the intersection of the cable-graph with the metric ball of radius $K$ centered at the origin do converge weakly. The reweighing here means that one has a Radon-Nikodym derivative which is equal to a constant (chosen so that one indeed gets a probability measure) times the square root of the field at the origin.  Note that the existence part of this theorem (i.e., without the very simple description of the limit in terms of the overlay of a loop-soup with the excursion) in the cases $\Z^d$ for $d \ge 3$ (except for $d=6$) was one recent main outcome of \cite {CD3} (as the title of that paper indicates!) which came on the shoulder of substantial works on one-arm estimates and used delicate and highly non-trivial exploration/renormalization arguments (see more about this and related questions in Section \ref {S2}).

\medbreak
{
We will also discuss switching results along closed loops, that are more reminiscent of the switching identities for the random current representation of the Ising model as initiated by Aizenman \cite {Ai}. Here we come back to the definition of the cable-graph based on a discrete graph (with sites $S$ and edges $E$) and view $G$ as the union of the segments corresponding to the edges of $E$.
We can note that each edge/segment will be crossed only finitely many times by loops in the Brownian loop-soup ${\mathcal L}$. The loops are not oriented, but we can make sense of the parity $P(e) \in \Z / (2 \Z)$ of the total number of crossings of an edge by ${\mathcal L}$ (i.e., the sum over all loops of ${\mathcal L}$ of their number of crossings of the edge); as opposed to the actual number of crossings of the edge, this parity remains unchanged if one glues two Brownian loops into one, or splits one Brownian loop into two loops. Let $E_\Lambda$ denote the set of edges along which $\Lambda$ remains strictly positive, and let $E(P)$ denote the set of edges for which $P$ is equal to $1$ (so this means that the number of crossings is odd). Clearly, $E(P) \subset E_\Lambda$ (because if $\Lambda$ hits $0$ on the edge, there can not be any crossing), and $(S,E(P))$ is an even subgraph of $(S,E_\Lambda)$ (i.e., for each site, the number of edges of $E(P)$ adjacent is even -- this is just because the crossings are made by loops).

\begin {proposition}
\label {Pnew}
Consider a Brownian loop-soup on the cable graph $G$ and let $\Lambda$ be its occupation time field. Conditionally on $\Lambda$, the graph $(S,E(P))$ is a uniform random even subgraph of $(S,E_\Lambda)$.
\end {proposition}
[When a graph is infinite but has only finite connected components, we say that a uniform even random subgraph is obtained by choosing independently a uniform random even subgraph for each component.]
This proposition is reminiscent of ideas from \cite {LW1}, but one important feature to stress here is that this result is valid conditionally on all of  $\Lambda$ (not just on $E_\Lambda$). Since the conditional law depends on $E_\Lambda$ only, this means that conditionally on $E_{\Lambda}$, the field $\Lambda$ and $P$ are independent. So, one can first observe the random set $E_\Lambda$, then choose independently the random even subgraph, and independently of the random even subgraph, the values of $\Lambda$ at the sites and then all the conditionally independent values of $\Lambda$ on each edge.
As we shall see, this corresponds to the fact that if one conditions on $\Lambda$, there is a  ``measure-preserving switch'' of the parity of all the edges along any given cycle of edges contained in $E_\Lambda$ (and this switch does not change $\Lambda$).

Let us describe one interesting concrete consequence of this proposition here: Consider the case of a cable graph that is contained in $\Z^d$ (and that we view as embedded in $\R^d$), and consider the $(d-2)$-dimensional space $\Delta := \{ 1/2\} \times \{ 1/2 \} \times \R^{\{ d-2 \}}$ (so when $d=2$, this is the point $(1/2, 1/2)$, when $d=3$, it is a line and so on). Then, each oriented loop in the cable graph has an index around $\Delta$ (which is the index around $(1/2, 1/2, 0,\ldots, 0)$ of its projection on the plane $\Z^2 \times \{0\}^{d-2}$). Whether this index is even or odd does not change when one changes the orientation of the loop. So, one can define unambiguously the total parity index $P_\Delta ( C, {\mathcal L})$ of the collection of loops in ${\mathcal L}$  that are part of a cluster $C$ of loops as the sum (in $\Z / 2\Z$) of the parity of the indexes of each of the Brownian loops in the cluster (which is finite as there are only finitely many loops that visit more than one site in $C$). Note that this quantity is again invariant under rewiring of the unoriented Brownian loops (i.e., it does not change when one concatenates or splits finitely many Brownian loops). Note finally that if $C$ does not contain any self-avoiding cycle of odd index around $\Delta$, then $P_\Delta (C, {\mathcal L})$ is necessarily equal to $0$.

\begin {corollary}
\label {Cnew}
Conditionally on $\Lambda$ (which contains the information defining all the loop-clusters), the collection $(P_\Delta (C, {\mathcal L}))$ for the loop soup clusters that do contain cycles with an odd index around $\Delta$ form an i.i.d. family of Bernoulli-(1/2) random variables (i.e., one tosses one fair coin for each of these clusters to choose its loop-soup index parity).
\end {corollary}

This result will work also if one replaces $\Delta$ by a simple compact $(d-2)$-dimensional manifold that does intersect the cable-graph.
One somewhat surprising consequence is that regardless of its shape, a loop-soup cluster that wraps around $\Delta$ (i.e., contains a cycle with non-zero index around $\Delta$) will have a conditional probability at least $1/2$ to actually contain a Brownian loop that wraps around $\Delta$ (indeed, if $P_\Delta ( C, {\mathcal L})=1$, then at least one of the Brownian loops in $C$ must has a non-zero index).
This will allow to prove Lupu's ``intensity doubling conjecture'' formulated in \cite {L3} (i.e., to prove that it holds) about loop-soups in high dimensions.
See Section \ref {Snew} (and \cite {LuW}) for more on all these aspects.}

\medbreak
Let us make the following first comments related to these statements:
\begin {enumerate}
\item
Theorem \ref {mainthm} easily implies Theorem \ref {thm1}: If one considers the former in the limit when $a\to 0$ and $b \to 0$, ones ends up conditioning $x$  and $y$ to be  boundary points of the same sign-cluster of $\Gamma$. In this case, the Poisson point process of excursions with intensity $a^2$ away from $x$ and the Poisson point process of excursions with intensity $b^2$ away from $y$ disappear in the limit, and the Poisson point process of excursions joining $x$ and $y$ with intensity $ab \nu_{x ,y, G}$ conditioned to be odd will in  the limit consist of one single excursion joining $x$ and $y$. We therefore readily obtain Theorem \ref {thm1} as a consequence of Theorem \ref {thm2}. One could also (we will do something like this in the context of interlacements in Section \ref {S2}) let only $a$ tend to $0$ while letting $b$ fixed, and one then also has only one excursion joining $x$ and $y$ in the limit (and no excursion away from $x$ -- one ends up with a description of the conditional law given that $x$ is on the boundary of the cluster containing $y$).
\item
Our statements can be easily shown to hold for any Brownian motion with killing rate instead of ordinary Brownian motion (the killing rate $k(x)$  can be dependent on the points $x$ on the cable-graph -- one can even pass to the limit where the killing rate becomes infinite at some points, which corresponds to transform those points into absorbing boundary points, i.e., introducing Dirichlet boundary conditions at this point) -- meaning that one modifies the measures on loops and the excursions accordingly. To see this, one option is to include this in the setup as in \cite {LJ2,LJ3} and see that it all goes through. Alternatively, one can deduce it from the case without killing (sometimes referred to as massless) and to use the observation that (for finite graphs and finite killing rates) all our switching identities are identities in distribution for the occupation time field. Passing from the massless to the massive case (i.e., when $k$ is not identically $0$) amounts just to a reweighing of the probability measure by a factor proportional to $U(\Lambda):= \exp (- \int \Lambda(x) k(x) dx)$ that depends on $\Lambda$ only. Hence, the switching identity is preserved under such reweighing. It also shows that it is not necessary for $k$ to be non-negative, as one just needs $U(\Lambda)$ to be integrable. We will comment on what the switching property says for long connections in the massive case in Section \ref {Smassive}.
\item
Introducing random walk representation of fields, lattice models and exploiting some of their combinatorial features is an idea that can be traced back (at least) to the seminal works of Symanzik or  Brydges, Fr\"ohlich and Spencer \cite {Sy,BFS}, see \cite {LJ3} for a recent relevant account. The present switching property can be reminiscent of the switching property for the random current representation of the Ising model from Griffith, Hurst and Sherman \cite {GHS} later developed and used very fruitfully by Aizenman and others, starting with \cite {Ai} (see the  ``the random current revolution'' section in the review \cite {DC}) where it is explained how  Ising correlations (and more) could therefore be expressed in terms of probabilities
involving multiple independent currents. We stress however that -- as is apparent from our GFF cable-graph statements here, where the events that one conditions on are {\em connection events on the cable graph} --  our results are very much ``cable-graph ones'' and do not seem to have an as simple counterpart for discrete loop-soups. On the other hand (see Section \ref {Snew}), it will be possible to relate aspects of our switching identities to those appearing in the random current representations of the Ising model (as the results of \cite {LW1} also suggest).
\item
We will provide several proofs. Let us say a few words about one rather ``self-contained'' proof to explain where the switching is coming from:
One first main step  will be to see that when conditioning on $\Lambda (x)$ and $\Lambda (y)$ only (and not on the fact that $x \leftrightarrow y$), then the decomposition of the Brownian loops that hit both $x$ and $y$ in the loop-soup into excursions away from $\{x,y\}$ combined with the rewiring property will yield a similar  decomposition as in Theorem \ref {thm2}, but where the final item is replaced by a Poisson point process of excursions with intensity $ab \nu_{x , y, G}$ that is {\em conditioned to be even} (instead of odd). Note that with this description, if one then additionally conditions on $x \leftrightarrow y$, the four inputs become very much correlated in the case where there is no excursion joining $x$ and $y$ (and when $x$ and $y$ are far away from each other, this will be the dominating event) since the event $x \leftrightarrow y$ will typically be created by an excursion away from $x$ that touches either an excursion away from $y$ or a chain of loops that then touches an excursion away from $y$. This description with even number of excursions does therefore not provide that much insight about connectivity events.

However, one can also look at the same conditional law for $\Lambda$ when one adds an extra ghost edge between $x$ and $y$ along which the GFF does not hit $0$. The Markov property of the GFF shows that one obtains the same conditional distribution for $\Lambda$ but with an extra weighting of $2$ for the configurations for which $x$ is not connected to $y$ in $G$ (as one does not need to toss a coin so that they have the same sign). So, this reweighted law can then be
described by the same four independent items but with an {\em unconditioned} (i.e., it can be even or odd) Poisson number of excursions joining $x$ and $y$ in $G$. Comparing this with the previous description then allows to conclude that
the probability of having an odd number of excursions is in fact be identical to the probability of having an even number of excursions and a connection between $x$ and $y$ (if one adds the first three items), and lead to the theorem.

\item
The theorem shows the existence of
a (measure-preserving) bijection that preserves the occupation times between the configurations with an odd number of excursions and the configurations with an even number of excursions that do create a connection between $x$ and $y$. This bijection is the  ``switching'' in the name that we gave to this property.
However, neither the statement nor the idea of proof that we just gave do provide an insight into what such a bijection might look like. Some thought does actually lead to the idea that such an explicit bijection cannot be obtained by finitely many reconnections/rewiring of loops/excursions,  due to  parity issues.  The arguments and results of \cite {QW,LQW} also provide food for thought in this direction. They indeed suggest that if the switching corresponds to some sort of rewiring on the loops and excursions, then some exceptional points that were not part of a chain of loops in the first instance will become part of an excursion (note that if these points are sufficiently sparse, then this won't affect the total occupation time).

In Section \ref {Snew}, we will discuss two aspects of these bijections: One at the level of edges (leaving a black box about what happens within each edge) that is closely related to the usual switching for random currents, and one at the level of the individual Brownian loops that hopefully enlightens what is going on here (even if the coupling is by no means trivial). This will build on the limit of reverse-vertex reinforced jump process that has been studied by Lupu, Tarr\`es and Sabot \cite {LST2}. This approach and its consequences will be detailed in subsequent work \cite {W4,LuW}.

\item
Both the GFF and its square (via the loop-soup representation) do satisfy a spatial Markov property. The switching property in some sense allows to reconcile the apparent parity-type contradiction that the combination of these two Markov properties do seem to create.  Of course, one can also view it the other way around (and this is basically the way the proof outlined above goes), namely that the switching property is a consequence of the combination of these two Markov properties. In some sense, all three proofs can be viewed as variations around this theme.

\end {enumerate}

  The paper is structured as follows: In Section \ref {S3}, we will describe a proof Theorem \ref {mainthm} that is fairly direct and self-contained, and builds mostly on loop-soup considerations and the rewiring property in particular. In Section \ref {SD}, we describe another approach/interpretation to this proof in the spirit of Dynkin's isomorphism theorem, so this involves more some Gaussian random variables. Section~\ref {Snew} provides yet another approach that highlights the link with the random current representation of the Ising model (and leads to Proposition \ref {Pnew} and to the proof of Lupu's intensity doubling conjecture). We also informally discuss in that section how to get an explicit bijections at the level of the Brownian loops.
   In Section \ref {S2}, we then describe the aforementioned direct consequences of the switching property (the constructions and descriptions of various versions of the infinite incipient clusters and statements about the combination of loop-soup with interlacements). In Section \ref {Se6}, we briefly outline some forthcoming work and results (respectively on the explicit bijection at the level of Brownian loops,  and on some further consequences about the asymptotic behaviour of large clusters, for instance on how to extract information about multiple-point functions).

\medbreak

Let us conclude this introduction with the following important remark: The proof (or the proofs, since we will provide several ones) of these switching identities are not technically difficult, and the underlying ideas bear many similarities with many of the papers in the area. In fact, once one realizes that the statements do hold, it is actually easy to work out proofs, building on various items in the existing literature -- and each of those acquainted with them will surely quickly be able to provide their own prefered approach (this includes Yves Le Jan or Lorca Heeney who independently sent me some compact versions after the posting of the first version of this paper, based on the results of \cite {LW1,LJ3}).
One can in fact detect aspects of this general switching property in a number of recent works on the GFF, especially those with the loup-soup perspective. This includes in the continuous and conformally invariant two-dimensional setting the papers \cite {QW,LQW} that we will further comment on, as well as \cite {JLQ} [Remark 1.12 in that paper can be viewed as the scaling limit i.e., as the continuum version of Theorem \ref {thm1} in this two-dimensional setting] or Lupu's paper on the twisted GFF \cite {L2}. In the cable-graph setting, as we shall discuss in Section \ref {S63}, an identity in law about bridges of three-dimensional Bessel processes by Pitman and Yor from the 1982 paper \cite {PY} can actually be interpreted as the switching property for the cable graph consisting of just one edge. Another relevant example is the paper \cite {Aid1} (see in particular Theorem 1.4 there that one could then use by adding an artificial additional edge bewtween two points to a cable graph) that also builds on the considerations \cite {W1,LW1} (as part of the present paper does) that describes the number of excursions along one edge in the cable graph, where the conditioned Poisson random variables appear. So, one possible short proof is to build on \cite {W1,LW1} to quickly deduce the switching for one edge, and then the arguments in Section \ref {Snew} that are more ``discrete switching'' arguments in the spirit of \cite {Ai2} and random current representations, to deduce the result in general cable graphs.

One can therefore wonder why such a striking general result that has many nice consequences has not been uncovered before. The main reason is arguably simply that the explicit bijection between configurations with even and with odd number of excursions at the level of the Brownian paths is not trivial (which provides an excuse to the community -- or at least to me -- to why it was  not so easy to guess the result), especially when one looks at it from the lens of Dynkin's isomorphism theorem without its loop-soup interpretation. Also, as we have already mentioned, it is not an observation that comes from discrete models -- ideas from the continuum (at least on an interval) are needed to get the ball rolling.
We note that it is not the first time that work on the (conformally invariant) two-dimensional continuous settings did lead mathematicians on the trail of  general basic results about the GFF and loop-soups on cable graphs: The loop-soup itself was first introduced in the continuum setting in \cite {LW} in relation to the scaling limit of loop-erased random walk and the conformal restriction properties developed in \cite {LSW}, and its relation to the continuum two-dimensional GFF (via the combination of \cite {SchSh,MS} and \cite {SW}) was derived before Le~Jan \cite {LJ1,LJ2} pointed out the direct relation between the square of the discrete GFF and the discrete loop-soups in general transient graphs or the direct basic relation with the loops erased in Wilson's algorithm was worked out (see e.g. \cite {L,WP}). This also holds in the present case: In \cite {LQW}, we were zooming in on some decompositions of two-dimensional loop-soup clusters  given part of their boundaries (following earlier considerations in \cite {QW,Q,QW2}), where some quizzing parity issues naturally popped up. A similar-looking switching identity appeared in \cite {LQW} [in the latter section of that paper, we derive an identity in the (more involved) continuous two-dimensional case and  for rectangular domains with some very special boundary conditions, so the points $x$ and $y$ in some sense correspond to two opposite vertical sides of the rectangle and the excursions from $x$ to $y$ do correspond to horizontal crossings]. This in turn led to the trail of the type of bijection that we will describe in Section \ref {Speeling} and also to the realisation that this type of  switching result was actually working already at the simpler level of cable-graph GFFs and not only for some specific values of the values $a$ and $b$. So, again a somewhat convoluted route -- with a special role played by the considerations, arguments and questions raised in \cite {QW,LQW} coming from the continuum two-dimensional world.

\section {A first proof of the switching property}
\label {S3}

We now turn to a first proof of Theorem \ref {mainthm}, which is the one that we outlined in the introduction.
Let us first write some words about the way in which we normalize the various objects involved: The GFF is defined via the Green's function of the cable-graph Brownian motion, which is the expected local time at $y$ cumulated by a Brownian motion started at $x$ until the possibly infinite time at which it exits the graph. The normalization of the Brownian loop measure is then the one such that the occupation time of a Poisson point process with that intensity is exactly distributed as the square of the GFF (this corresponds to $c=1$ in the normalization/notation of \cite {LW} inspired by the notion of central charge for two dimensional models, or $\alpha = 1/2$ in the papers by Le Jan and Lupu).
The excursion measure $\nu_{x,x}= \nu_{x,x, G}$ away from a point $x$ in $G$ is chosen in such a way that $ \nu_x ( \ell_e(z) ) = \phi_x (z)^2$ where $\ell_e$ is the local time of the excursion $e$ at $z$ and $\phi_x(z)$ is the harmonic function in $G \setminus \{ x\}$ with boundary values $1$ at $x$ and $0$ at all other boundary points (and at infinity). The excursion measure $\nu_{x , y}$ between two points can then be constructed as the measure on paths joining $x$ to $y$, obtained by first restricting the measure $\nu_{x,x}$ to those paths that hit $y$, and then keeping only the part of this path up until the first time it hits $y$.

It is also worth recalling that (when viewed in relation to the GFF and to the rewiring ideas of \cite {W1} that will play an important role here) the measures on loops and on excursions are in fact more naturally defined on ``non-oriented'' paths, meaning that the paths are defined up to time-reversal (the path joining $x$ and $y$ can be viewed alternatively as from $x$ to $y$ or as from $y$ to $x$). Recall also that Brownian loop-measures are similarly  most naturally viewed as measures on unrooted loops -- see for instance \cite {LW,W1,WP}.

We will try to keep the narrative in this proof as intuitive as possible, in order to emphasize how it can be reduced to loop-soup decompositions in the discrete setting. The reader favouring compact analytic proofs may view this entire section as a warm-up to the second proof that we present in Section \ref {SD}.

\medbreak

Let us first briefly recall two different ways to approach the conditional distribution of the GFF in $G$ given
 $\{ \Gamma (x) = a, \Gamma (y)= b \}$ for $a, b > 0$. One will be to condition directly on $\Gamma$. The second one will be to first discover $\Lambda = \Gamma^2$ via the loop soup, and to then choose the signs independently.
The idea to combine and compare these two ways has been used several times when deriving properties of loop-soup clusters (see e.g. \cite {MSW,MW} in the $\kappa=4$ cases, or \cite {L2,QW,LQW}).

\begin {enumerate}
\item The first option  is to view
$\{ x, y\}$ as part of the (same wired) boundary (which is possible, as $a$ and $b$ have the same sign), and to use Dynkin's isomorphism. The conditional distribution of $\Gamma$ is then the sum of the GFF in $G \setminus \{ x, y \}$ with the harmonic function in that set with boundary values $a$ and $b$ at $x$ and $y$ respectively.
By Dynkin's theorem, the conditional distribution of the square of the GFF will be given by the sum of the following independent inputs:

- The square of the GFF with zero boundary conditions in $G \setminus \{x,y\}$.

- The occupation time of an (independent) Poisson point process of excursions away from the set $\{ x, y \}$ with well chosen intensities.
This Poisson point process of excursions will contain:
\begin {itemize}
 \item A Poisson point process of excursions from $x$ to $x$ in $G \setminus \{ y\}$ (with infinitely many small ones) with intensity $a^2 \nu_{x, x, G \setminus \{y\}}$ and a Poisson point process of excursions from $y$ to $y$ in $G \setminus \{x \}$ with intensity $b^2 \nu_{y,y, G \setminus \{x\}}$.
 \item A Poisson point process of excursions joining $x$ and $y$ with  $ab \nu_{x , y, G}$. The finite number $N=N_{x,y}$ of such  excursions will therefore be a Poisson random variable with mean given by the total mass of this measure.
 \end {itemize}
Note that in this picture, the excursions from $x$ to $x$ are not allowed to visit $y$ (and the ones from $y$ to $y$ are not allowed to visit $x$). Mind also that in this description, there is no parity constraint on $N$. Note finally that if $N \ge 1$, then $x$ and $y$ are necessarily in the same sign-cluster. On the other hand, when $N=0$, both options $x \leftrightarrow y$ and $x \not\leftrightarrow y$ are possible, depending on whether the event that at least one sign-cluster of the zero-boundary GFF does intersect both an excursion away from $x$ and an excursion away from $y$ occurs or not. In particular, when one conditions on $\{ N = 0 \} \cap \{ x \not\leftrightarrow y \}$, then the various excursions and loops are not independent.

\item
Alternatively, one can start with the loop-soup in $G$, and decompose it into the loops that do hit $\{x, y\}$ and the ones that don't. The occupation time  of the latter part will give rise to the same square of GFF in $G \setminus \{ x,y\}$ with zero boundary conditions as in (1). The loops that hit $x$ but not $y$ will give rise to (many) excursions away from $x$ that do not hit $y$, the loops that hit $y$ but not $x$ will give excursions away from $y$ that do not hit $x$.  The loops that hit both $x$ and $y$ will give rise to excursions away from $\{ x, y \}$ of three types: Excursions away from $x$ that do not hit $y$, excursions away from $y$ that do not hit $x$ and   a necessarily even number of excursions joining $x$ and $y$. The resampling and rewiring ideas suggest that when conditioned on the number of excursions joining $x$ and $y$, they will be distributed like independently chosen Brownian excursions. Similarly, when conditioned on the total local times at $x$ and $y$ to be $a^2$ and $b^2$ respectively, then the excursions away from $x$ and away from $y$ should end up been chosen according to a Poisson point process of excursions as in Description (1).

Once one has constructed the loop-soup, one can then define the GFF by choosing a sign independently for each loop-soup cluster. Hence, the conditional probability that $\{\Gamma (x)=a,  \Gamma (y)=b\}$ will be $1/2$ or $1/4$ depending on whether $x \leftrightarrow y$ or not.
\end {enumerate}

The first step of the proof is the following fact that gives the complete picture for Description (2):
\begin {lemma}[The parity lemma]
\label {parity1}
In Description (2), conditionally on $\Lambda (x)= a^2$ and $\Lambda (y) = b^2$, the decomposition of the loops intersecting $\{ x, y\}$ into excursions away from $\{ x ,y\}$ leads to the very same decomposition as  Description (1), except that the number of excursions  $N_{x \leftrightarrow y}$ joining $x$ to $y$ is now {\em conditioned to be even}.
\end {lemma}

Actually, a more general statement with $n$ points $x_1, \ldots, x_n$ holds: The conditioning would then be that if $N_{i,j} = N_{j,i}$ is the number of excursions joining $x_i$ to $x_j$ in $G \setminus \{ x_1, \dots, x_n\}$, then $(N_{i,j})_{1 \le i < j \le n}$ has the law of independent Poisson random variables with respective means $a_ia_j$ times the mass of the corresponding excursion measure, where this collection is conditioned by the event that for all $i$, $\sum_{j \not=i} N_{i,j}$ is even. In fact, this general result (and therefore also this parity lemma) is almost exactly Proposition 7 in \cite {W1}. The one slight difference in the formulation there is also that it describes only the conditional law of the excursions from $x$ to $y$ (or from any $x_i$ to any $x_j$ for $i \not= j$), and does not say anything about the excursions from one of the points to itself. It is however not difficult (for instance building on the discrete setting) to see that these point processes of excursions will be independent (and of course the same for the loops that do not intersect the marked points, since they are anyway independent from $(\Lambda(x_i))_{i \le n}$).
Since the details of the proof there were ``left as a simple exercise'' (i.e., the part that derives it from its analog for discrete time loop-soups),
we will provide some more details about this proof.

But before that, let us first  show how to use it in order to deduce the switching result. The basic idea is to notice that, when one conditions on $x \leftrightarrow y$, then the law of the configuration described in (1) i.e., via an unconditioned Poisson random variable $N$, is the same as the law of the configuration obtained via (2), i.e., when the Poisson random variable is conditioned to be even.
As a consequence, the same is true when one conditions on the remaining event $\{x \leftrightarrow y\} \setminus \{ N \hbox { is even}\}$, i.e., when $N$ is conditioned to be odd.

\begin {proof}[Proof of Theorem \ref {mainthm} using Lemma \ref {parity1}]
We fix $a$ and $b$ as in the theorem, and consider the setup of the first description above (where the number of excursions joining $x$ and $y$ is an unconditioned Poisson random variable). The probability measure $P$ will correspond to this setup.
We let $m$ be the mean value of the Poisson random variable $N=N_{x \leftrightarrow y}$ (which is the $ab$ times the total mass of $\nu_{x ,y, G}$) in the first description. We denote by:
\begin {itemize}
 \item $E$ and $O$ the events that $N$ is even or odd respectively.
 \item  $Q_-$ the event that $x$ and $y$ are not in the same sign-cluster.
 \item $E_+$ the event that $N$ is even and $x$ and $y$ are in the same sign-cluster.
\end {itemize}

Obviously, since $N$ is a Poisson random variable with mean $m$, we have that
$$P [ E] = \frac 1 2  ( 1 + e^{-2m}), \ P[O] = \frac 1 2 ( 1- e^{-2m})
\hbox { and }
P [ E ] - P [O] = e^{-2m}.$$

In  Description (2), when one conditions the total local time at $x$ and $y$ to be $a^2$ and $b^2$, one does not yet know whether the signs of $\Gamma (x)$ and $\Gamma (y)$ are the same. So, in order to get the same law as in Description (1), one has to reweigh the event $Q_-$ by a factor $1/2$ because on $Q_-$, the conditional probability that $\Gamma(x)$ and $\Gamma(y)$ have the same sign is $1/2$. In other words, if $P_E$ denotes the law of the configuration (with excursions and loops) from Description (2) where $N$ is conditioned to be even, we get that for any event $A$ depending on the occupation field,
\begin {equation}
 \label {eq1}
 P[ A] = \frac 1 {1 - (P_E [Q_-]/2)} \times
\left( P_E [ A \setminus Q_-] + \frac {P_E [A \cap Q_-]}{2}\right).
\end {equation}
In particular, for $A = Q_-$, we get that
$$ P [ Q_-] = \frac { P_E [Q_-]}{ 2 - P_E [Q_-]}.$$
But $P_E [ Q_-] = P [ Q_-] / P [E]$, so that
$$P[Q_-] = 2P[E]- 1 =  e^{-2m} = \exp ( -2m) = P [E] - P [ O].$$
Note that
$$P [E \setminus Q_- ] = P[E] - P[Q_-] = 1- P [E] = P [O] .$$
So, we can conclude that if one constructs a conditioned GFF via Description (1),
$$ P [ E \mid x \leftrightarrow y , \Gamma^2(x) = a^2, \Gamma^2 (y) = b^2 ] = P [ O \mid x \leftrightarrow y , \Gamma^2(x) = a^2, \Gamma^2 (y) = b^2 ] = \frac {1}{2}.$$

We now rewrite (\ref {eq1}) this time for $A \subset \{ x \leftrightarrow y\}$ such that $P[A] \not= 0 $:
$$
\frac {P_E [ A]}{P [A]}  =   1 - \frac {P_E [Q_-]}{2} .$$
We note that this quantity is independent of $A$. In particular, it holds also for the set $x \leftrightarrow y$ itself, so that
$$
\frac {P_E[A]}{P_E [ x \leftrightarrow y]  } =
\frac {P[A]}{P [ x \leftrightarrow y]}.$$
This last identity of course holds also when $P[A]=0$. We can therefore conclude that the conditional distribution of the (conditional) field (we always work conditionally on $\Gamma (x)=a$ and $\Gamma (y) = b$) given that $N$ is even and $x \leftrightarrow y$ is the same as its conditional distribution given $x \leftrightarrow y$ only. It follows that it is also the same as its conditional distribution given that $x \leftrightarrow y$ and $N$ is odd (which is just the conditional distribution given that $N$ is odd), which completes the proof.
\end {proof}

\begin {remark}
We note that as in \cite {LQW}, the proof does not only provide the identity of the conditional distributions, but also the fact that in the first description, the contributions of $P[E \cap \{ x \leftrightarrow y \}]$ and $P[O]$ to $P[x \leftrightarrow y]$ are equal, so that there exists a measure-preserving bijection (under $P$) between $E \cap \{ x \leftrightarrow y\}$ and $O$ that preserves also the occupation times.
\end {remark}

We now turn our attention to the proof of Lemma \ref {parity1}: We will detail the proof outlined in \cite {W1}. A first step is to view this as
 a   result for loop-soups of a simple continuous-time Markov chain $X$ with two states $x$ and $y$ (and a cemetery boundary state $\partial$) defined as follows: When at $x$, the chain jumps to $y$ with rate $\alpha>0$ and it dies with rate $u_x$. When at $y$, the chain jumps to $x$ with the same rate $\alpha$ and it dies with rate $u_y$. Indeed, when following the trace of the cable-graph Brownian motion on the set $\{ x , y \}$, which is all one is interested when counting its excursions between $x$ and $y$, one has exactly a Markov chain of this type (the time of the latter is the local time accumulated by the former at $x$ or $y$).

We consider the (critical) loop-soup ${\mathcal L}$ corresponding to this  continuous-time chain on $\{ x, y\}$, as in the settings introduced by Le Jan (see for instance \cite {WP} for a description of such continuous-time discrete loop-soups). So, the loops that visit both $x$ and $y$ have exponential waiting times at each visit of $x$ and $y$, while the loops that visit only one point have time-lengths distributed according to an (infinite) measure (i.e., there will be infinitely many such small stationary loops). The occupation time measure of this loop-soup is then distributed as the square of the GFF associated to the Green's function of this Markov chain.
The relation between the cable-graph Brownian motion and this continuous-time discrete loop-soup shows that one can view this loop-soup as the trace on $\{x,y\}$ of the loop soup on the cable graph $G$. So, part of the parity lemma will boil down to the following fact:

\begin {lemma}
\label {parity2}
Conditionally on $\Lambda_{\mathcal L} (x) = a^2$ and $\Lambda_{\mathcal L}(y) = b^2$, the number of jumps between the two points $y$ and $x$ in ${\mathcal L}$ is distributed as a Poisson random variable with intensity $\alpha ab$ {\em conditioned to be even}.
\end {lemma}

\begin {proof}[Proof of Lemma \ref {parity2}]
Our proof will be based on the ideas of  Proposition 2.46 from \cite {WP} or \cite {W1} (see also \cite {L} for analogous statements), which gives the explicit expression for the number of jumps on edges for discrete loop-soups (with discrete time).
We each (large) constant  $K$, we define the discrete-time Markov chain $X^K$ on the same set $\{ x, y ,\partial\}$ with  jump probabilities  $p(x,y) = p ( y,x) = \alpha/K$, $p (x,x) = 1-  (u_x + \alpha)/ K$, $p(x, \partial)
= u_x / K$, $p(y, \partial) = u_y / K$, $p(y, y) = 1 - (u_y + \alpha)/K$ (the point $\partial$ is an absorbing state). In the limit when $K \to \infty$, the number of steps of this discrete chain $X^K$ divided by $K$ then corresponds to the time for the continuous-time Markov chain $X$.

The first step is to notice, just as in Proposition 2.46 of \cite {WP}, that for a discrete loop-soup on this graph, the probability for having $2t$ jumps joining $x$ and $y$ (this number has to be even since each loops jumps an even number of times on this edge) with $A$ and $B$ total visits of $x$ and $y$ is a constant multiple of
$$ p(x,x)^{A-t} p(y,y)^{B-t} p(x,y)^t p(y,x)^t
\frac{ {\mathcal P}(2A) {\mathcal P}(2B)}{t! (A-t)! (B-t)!}$$
where ${\mathcal P} (2n)=(2n-1)(2n-3) \ldots 3$ denotes the number of possible pairings of a set with $2n$ points.
In particular, if one conditions on $A$ and $B$, the conditional
probability of having $t$ jumps between $x$ and $y$ is a constant (depending on $A$ and $B$) times
\begin {equation}
 \label {edgeprobab}
p(y,y)^{-t} p(x,x)^{-t} p(x,y)^{2t} \frac {1}{t! (A-t)! (B-t)!}.
\end {equation}
One simple way to see this is to replace each edge into a large number $M$ of ``parallel'' edges, so that with high probability, no edge is used more than once by the discrete loop-soup. One can then enumerate the number of such possibilities that do not use any edge more than once, and then letting $M \to \infty$ gives (\ref {edgeprobab}).

We are now going to apply this for $X^K$ and let $K \to \infty$.
It is easy to see one the one hand that in this limit, the discrete loop-measure will converge to the continuous-time loop measure (with the renormalized number of steps converging to continuous time). One can indeed simply couple everything with the cable-graph loop-soup in $G$. One can for instance for $\eps = \eps (K)$, discretize the cable-graph Brownian motion $B$ started from $x$ as follows: One chooses $\tau_0= 0$, and then $\tau_1$ to be the first time at which either the local time at $x$ reaches $\eps$ or $B$ reaches $y$ or $\partial$. Then, $X_1 = B_{\tau_1}$. Then, similarly, for each $n \ge 1$, one lets $\tau_{n+1}$ be the first time after $\tau_n$ at which either the local time at $X_{n}$ has increased by $\eps$, or $B$ has reached $\{ x, y , \partial \} \setminus \{ X_n \}$. One then defines $X_{n+1} = B_{\tau_{n+1}}$.   In this way (when $\eps(K)$ is well-chosen) one indeed has exactly the  jump probabilities of $X^K$. The definition of local times shows also that renormalizing the number of steps of $X^K$ converges almost surely to the local times (at $x$ and $y$) of $B$ when $K \to \infty$.

On the other hand, letting formally $K \to \infty$ in Formula (\ref {edgeprobab}) with $A=A(K) \sim a^2 K$ and $B=B(K) \sim b^2 K$, we see that that the conditional probability to jump $2t$ times along the edge joining $x$ and $y$ tends to a constant (depending on $\alpha$, $a$ and $b$ only) multiple of $ \alpha^{2t} (a^2 b^2)^t / (2t)!$ because
$$ p(y,y)^{-t} p(x,x)^{-t} p(x,y)^{2t} \frac {1}{t! (A-t)! (B-t)!} \sim 1 \times 1 \times (\alpha/K)^{2t} \frac {(a^2K)^t (b^2K)^t}{A! B!} \sim \frac {\alpha^{2t} a^{2t} b^{2t}}{A! B!}.$$
It is then easy to deduce the statement of the lemma. To control the conditioning on the occupation times in the scaling limit, one can for instance first see that for any $0< a_1 < a_2$ and $0 < b_1 < b_2$, any subsequential limit when $K \to \infty$ of the law of the number of jumps $N$ along the edge joining $x$ and $y$ when the conditioning is on $\Lambda (x) \in (a_1^2, a_2^2)$ and $\Lambda(x) \in ( b_1^2, b_2^2)$ can be coupled with Poisson random variables $P_1$ and $P_2$ with respective means $a_1b_1 \alpha$ and $a_2b_2 \alpha$ and conditioned to be even in such a way that $P_1 \le N \le P_2$.
\end {proof}

We can now finally derive Lemma \ref {parity1}, which will also conclude the proof of Theorem \ref {thm2}.
\begin {proof}[Proof of Lemma \ref {parity1}]
One notices that
in the coupling between the cable-graph loop-soup and the continuous-time loop soup on the discrete graph, all items are in one-to-one correspondence (the loops of the former that go through $x$ and/or $y$ are exactly loops of the latter, and the times of the latter correspond to the local times of the former). Furthermore, in the aforementioned coupling between the Brownian motion and the Markov process $X$, the limit when $\eps \to 0$ of the parts of the Brownian paths corresponding to the jumps from $x$ to itself will converge to a Poisson point process of excursions away from $x$ in $G \setminus \{ y \}$ with intensity
$a^2 \nu_{x,x,G \setminus \{ y\} }$, and the corresponding result will hold for the parts corresponding to the jumps from $y$ to $y$.
\end {proof}

\begin {remark}
 \label {compact}
We can note that Theorems \ref {thm1} and \ref {thm2} both deal with loop-soups that are conditioned on their occupation time at a given point (or alternatively the GFF conditioned to have a certain finite value at a given point). For instance, Theorem \ref {thm1} can be viewed as a statement about the loop-soup in $G \setminus \{x\}$ or alternatively about the GFF in $G \setminus \{ x \}$ which is well-defined also when Green's function on $G$ is infinite (for instance, when the graph $G$ is compact with no boundary point, or the case $G = \Z^2$) because the Green's function in $G \setminus \{ x\}$ is automatically finite.  Similarly, conditioning on the value of $\Lambda (x)$ to take a given positive value $a^2$ amounts (on the loop-soup side) to adding a Poisson collection of excursions away from $x$ (and possibly to regroup them into loops if one wants to describe this as a loop-soup), or to consider the GFF in $G \setminus \{ x \}$ with boundary condition $\pm a$ at $x$.
So, it is possible to make sense of the Brownian loop-soup in $G$ conditioned on $\Lambda (x) =a^2$ also when the Green's function on $G$ is infinite, and to then see that Theorem \ref {thm2} and its proof still hold in the same way, i.e., that Theorems \ref {thm1} and \ref {thm2} are actually valid for any cable-graph.
\end {remark}

\begin {remark}
The paper \cite {Aid1} by Elie A\"\i d\'ekon contains some closely related considerations, see in particular his Theorem 1.4 where parity shows up -- it does not address the switching property in terms of conditioning on long connections, but one a posteriori detect the switching property hiding near-by -- one can for instance try to see what happens when one adds an additional edge to his setting. As we shall see, he also looked into the direction that we will describe in Section \ref {Se5}.
\end {remark}

\begin {remark}
 \label {edgetograph}
 We can note that with this approach, the proof for a general $G$ does not differ much from the case where the only connection between $x$ and $y$ goes along a single edge of the graph -- replacing the entire graph by a single edge with the same ``effective resistance'' (which corresponds to ideas developed in \cite {LW3}) works well in a bijective way, both in terms of excursions and loop-soups (and the parity of excursions between $x$ and $y$), the connection via $\Lambda$ or the GFF correlation functions.
\end {remark}

\section {The derivation in the spirit of Dynkin's isomorphism}
\label {SD}

\subsection {The proof}

Let us now explain how to approach the proof of the switching property via explicit computations reminiscent of Dynkin's isomorphism, i.e., how to interpret the various terms appearing in those computations in terms of Poisson point process of excursions with even or odd numbers (since this is what is what the outcome of the parity and switching properties are).

The (by now) classical idea is that one can compute explicitly the Laplace transforms for the occupation time fields of all quantities involved (loop-soups, conditioned loop-soups, excursions).  In that setting, the switching will appear via the appropriate interpretation of the terms in the expansion of the exponential.
While this approach seems rather compact and of course in some sense equivalent to the previous one, it is possibly a bit less transparent!
Many aspects of this description will be reminiscent of our paper \cite {LQW} with Matthis Lehmk\"uhler and Wei Qian, see Section \ref {lqw}.

Let us consider a cable-graph GFF $\Gamma$ on $G$. When ${\mathcal M}$ is a random non-negative continuous function on $G$, we can define, for each non-negative function $k$ with compact support, the quantity
$$ {\mathcal E}_k ( {\mathcal M} ) := E [ \exp ( - \int k(z) {\mathcal M} (z) dz ) ]$$
(here and in the remainder of this section, $dx$ denotes the Lebesgue measure on the cable graph and the integrals will always be over all points in $G$).
The knowledge of ${\mathcal E}_k ( {\mathcal M})$ for all such functions $k$ clearly characterizes the law of ${\mathcal M}$ uniquely. On the other hand, for the random measures that are defined as occupation times of loop-soups or excursions, the quantity ${\mathcal E}_k [ {\mathcal M}]$ can be related to the corresponding quantities for loop-soups or excursions associated with the Brownian motion with killing rate given by $k$ (i.e., the Brownian motion is killed at rate $k(x)$ when it is at $x$), leading to expressions involving the Green's function for this new process. This idea lies at the core of Le~Jan's results \cite {LJ1,LJ2} (see e.g., \cite {LJ3,WP} for surveys).

\subsubsection* {Revisiting Dynkin's isomorphism}

Let us (for notational convenience) consider our two special points to be $x_1$ and $x_2$. We will fix $a_1, a_2 > 0$
Suppose now that $\Phi$ is a harmonic function on $G \setminus \{x_1, x_2 \}$, which is continuous at $x_1$ and $x_2$.  In this paragraph, $\Gamma_0$ will denote the GFF with zero boundary conditions at $\{ x_1, x_2 \}$, so that $\Gamma_0 + \Phi$ is in fact a GFF in $G$ conditioned on its values at $x_1$ and $x_2$ to be $\Phi(x_1)$ and $\Phi(x_2)$.  We can consider the field $(\Gamma_0 + \Phi)^2$ and simply expand
$(\Gamma_0+\Phi)^2 =
 \Gamma_0^2   + 2 \Gamma_0   \Phi + \Phi^2 $, and we can rewrite ${\mathcal E}_k ((\Gamma_0 + \Phi)^2)$ as
\begin {eqnarray*}
 \lefteqn {{\mathcal E}_k ((\Gamma_0 + \Phi)^2)}\\
&&={\mathcal E}_k ((\Gamma_0 )^2) \times
\exp ( -   \int \Phi^2 (z) k(z) dz ) \times E\left[ \exp (- 2 \int \Gamma_k (x) \Phi(x) k(x) dx ) \right],
\end {eqnarray*}
where $\Gamma_k$ is the Gaussian free field in $G \setminus \{ x_1, x_2\}$ associated with the Brownian motion with killing rate $k$, i.e., the GFF $\Gamma_0$  but with law reweighed via the Radon-Nikodym derivative term $\exp ( - \int \Gamma_0^2 (x) k(x) dx )$. This new GFF has a covariance function given by the Green's function $G_k$ of this Brownian motion with killing.
By inspecting the variance of the centered Gaussian variable $\int \Gamma_k (x) k(x)
\Phi (x) dx$, we see that the final term in the product has the explicit form
$$ \exp (  \int \int G_k (x, y) k(x) k(y) \Phi(x) \Phi (y)  dx dy ),$$ so that
$$
  {{\mathcal E}_k ((\Gamma_0 + \Phi)^2)}=
  {{\mathcal E}_k ((\Gamma_0)^2)} \times
\exp ( -   \int \Phi^2 (z) k(z) dz  + \int \int G_k (x, y) k(x) k(y) \Phi(x) \Phi (y)  dx dy )
.$$

Let us now consider the  two harmonic functions $\Phi_1$ and $\Phi_2$ in $G \setminus \{ x_1, x_2\}$ with respective boundary conditions $a_1 1_{x_1}$ and $a_2 1_{x_2}$ on $\{ x_1, x_2 \}$ (and that go to $0$ at infinity if $G$ is unbounded). The previous identity in the case of $\Phi = \Phi_1$ can be reinterpreted in terms of the loop-soup as follows:
\begin {itemize}
 \item The term
 $$ (1) := {\mathcal E}_k ((\Gamma_0 )^2)$$  corresponds the (Laplace transform of the) occupation time of a loop-soup in $G \setminus \{ x_1, x_2 \}$.
 \item The term
 $$(2)_i := \exp ( -   \int \Phi_i^2 (z) k(z) dz )  \times
 \exp ( \int\int  G_k (x, y) k(x) k(y) \Phi_i(x) \Phi_i (y)  dx dy )$$ for $i=1$
 corresponds to the (Laplace transform of the) occupation time of the Poisson point process of excursions away from $x_1$ in $G \setminus \{ x_1, x_2 \}$ (with intensity $a_1^2$ times the normalized excursion measure in that set).
\end {itemize}
We can also apply the same reasoning to $\Phi_2$ and then finally to $\Phi =\Phi_1 + \Phi_2$, and then comparing the obtained expression with the ones obtained separately for $\Phi_1$ and
$\Phi_2$, we can finally interpret the cross term
$$ (3) := \exp ( -  2 \int \Phi_1  (z) \Phi_2 (z) k(z) dz )  \times
 \exp (  2 \int \int G_k (x, y) k(x) k(y) \Phi_1(x) \Phi_2 (y)  dx dy )$$ as corresponding to the occupation time of the Poisson point process of excursions joining $x_1$ and $x_2$ in $G \setminus \{ x_1, x_2 \}$ with intensity $a_1 a_2 \nu_{x_1, x_2}$. So, we recognize in the expansion of
 $$ {\mathcal E}_k( ( \Gamma_0 + \Phi_1 + \Phi_2)^2 ) = (1) \times (2)_1 \times (2)_2 \times (3)$$ the four independent contributions (loop-soup, excursions away from $x_1$, excursions away from $x_2$ and Poisson point process of excursions joining $x_1$ and $x_2$) of Description (1)  in Section \ref {S3}.

 Mind that the expansion for $\Phi=\Phi_1 - \Phi_2$ does not have such a nice interpretation because of the different sign of the cross-term. However, all three other contributions are exactly the same, while the cross-term now simply gets an extra minus sign in the exponential, i.e., one has
 $$ {\mathcal E}_k( ( \Gamma_0 + \Phi_1 - \Phi_2)^2 ) = (1) \times (2)_1 \times (2)_2 \times(1/ (3)).$$  So when adding or subtracting this to the $\Phi_1+\Phi_2$ case will lead to factorizations, and as we shall now see we will get the parity lemma by interpreting the (appropriately weighted) sum of the two cases and the switching lemma by interpreting the (appropriately weighted) difference of the two cases.

Let us finally note that if one considers a Poisson point process of excursions joining $x_1$ and $x_2$ with intensity $a_1 a_2 \nu_{x_1,x_2}$, and $T$ its occupation field, then decomposing according to the Poisson number of excursions (we will from now on denote the total mass of $a_1 a_2 \nu_{x_1, x_2}$ by $m$), we get that
$$
(3) = E [ \exp (- \int  T(x) k(x) dx  )]
=\sum_{n \ge 0} e^{- m} \frac { (  \nu_{x_1,x_2} ( \exp ( - \int T_e (x) k(x) dx)))^{n}} {n!}
= \exp (  m^{(k)} -  m), $$
where $m^{(k)}$ is now the total mass of the excursions from $x_1$ to $x_2$ for the Brownian motion with killing rate $k$ (and $T_e$ denotes the occupation time field of the one excursion $e$ defined under $\nu$).
In view of what follows, we can also notice that if we condition the Poisson point process to have an even number of excursions, we get
$$
(3)':= \frac { \cosh (  m^{(k)}  )}{ \cosh ( m)},$$
while if we condition the process to have an odd number of excursions, we get
$$(3)'' := \frac { \sinh (  m^{(k) })}{ \sinh ( m)}.$$

\subsubsection* {Proof of the parity lemma}
Let $\Gamma$ be a GFF in $G$.
Let us define $\sigma_i$ to be the sign of $\Gamma(x_i)$. By simply inspecting the joint law (and density functions) of the Gaussian vector $(\Gamma(x_1), \Gamma (x_2))$, it is easy to check that
$$ p:=  P [ \sigma_1 =  \sigma_2  \mid  \Gamma(\partial_1)^2 = a_1^2, \Gamma (\partial_n)^2 = a_2^2 ]
= \frac {e^{  m}}{e^{ m} + e^{ -  m}}
$$
[this conditional probability is  described in terms of the densities of the law of the Gaussian vector $(\Gamma (x_1), \Gamma (x_2))$ at $(a_1, a_2)$ and $(a_1, -a_1)$ (i.e., in terms of $G(x_i, x_j)$ for $1 \le i \le j \le 2$), so that one just needs to relate this in terms of the excursion measures, which we safely leave to the reader].

If we now condition only on the square of the GFF at $x_1$ and $x_2$, we can then first condition on the signs of the GFF at these points, and then add the two contributions corresponding to the options $\Phi_1+\Phi_2$ and $\Phi_1 - \Phi_2$. In other words, (and the readers acustomed to the random current ideas may recognize the cancellations of the odd terms in the expansion of the exponential here):
\begin {eqnarray*}
 \lefteqn { E[ \exp ( - \int \Gamma^2 (x) k(x) dx) \mid \Gamma^2 (x_1) = a_1^2 , \Gamma^2 (x_2) = a_2^2 ]}\\
 &=& \left( p \times (1) \times (2)_1 \times (2)_2 \times (3)\right)  + \left((1-p) \times (1) \times (2)_1 \times (2)_2 \times (1/ (3)) \right)  \\
 &=& (1)\times (2)_1 \times (2)_2  \times   \left( \frac {e^{ m}}{e^{ m} + e^{ -  m}}
 \exp ( { m^{(k)} -  m})
 + \frac {e^{- m}}{e^{ m} + e^{ -  m}}
 \exp ({ -  m^{(k)} +  m})  \right)
 \\&=&
 (1) \times (2)_1 \times (2)_2 \times  \frac {\cosh (  m^{(k)})} { \cosh ( m)} .
\end {eqnarray*}
We recognize the last term as $(3)'$, corresponding to the Poisson point process of excursions joining $x_1$ and $x_2$  conditioned to have an even number of excursions. This is exactly the parity lemma.

\subsubsection* {Proof of the switching property}

This is  where we are going to use the fact that one can first sample the entire square of the GFF and then choose the signs of the sign-clusters independently. In that setting, we know that the contributions to the occupation times coming from $\Phi_1 - \Phi_2$ will be exactly the same as the one coming from $\Phi_1 + \Phi_2$ on the event $x \not\leftrightarrow y$.

In other words, we are this time getting the difference (instead of the sum) of the two same terms as above:
\begin {eqnarray*}
 \lefteqn { E[ 1_{x_1 \leftrightarrow x_2} \exp ( - \int \Gamma^2 (x) k(x) dx) \mid \Gamma^2 (x_1) = a_1^2 , \Gamma^2 (x_2) = a_2^2 ]} \\
 &=& (1) \times (2)_1 \times (2)_2  \times   \left( \frac {e^{ m}}{e^{ m} + e^{ -  m}} \exp (  m^{(k)} -  m)
 -  \frac {e^{- m}}{e^{ m} -  e^{ -  m}}
 \exp ( -  m^{(k)} +  m) \right)\\
 &=&
 (1) \times (2)_1 \times (2)_2 \times  \frac {\sinh (  m^{(k)})} { \cosh ( m)} .
\end {eqnarray*}
We recognize in the last term a multiple of  $(3)''$ corresponding the Poisson point process of excursions joining $x_1$ and $x_2$  conditioned to have an odd number of excursions (the remaining multiplicative term $\tanh ( m)$ is exactly the conditional probability that $x_1 \leftrightarrow x_2$, corresponding to the expression when $k=0$). So, we can conclude that
$$
E[ \exp ( - \int \Gamma^2 (x) k(x) dx) \mid \Gamma^2 (x_1) = a_1^2 , \Gamma^2 (x_2) = a_2^2 , x_1 \leftrightarrow x_2 ]
= (1) \times (2)_1 \times (2)_2 \times (3)''$$
which is exactly the switching property.

\subsection {Relation with the statements in \cite {LQW}}
\label {lqw}
Let us say a few more words about Theorem~2 and Theorem~6 of \cite {LQW}, their motivations and proofs. These results deal with the following set-up in the continuum: One considers a rectangle $R$ in the plane, a critical Brownian loop-soup in $R$, and a Poisson point process of excursions with some given fixed intensity away from the union of the two vertical sides of the rectangle (so this process will contain infinitely many small excursions away from the left boundary, infinitely many excursions away from the right boundary, and a Poisson number $N$ of excursions joining the left and the right boundary). Theorem 2 then says that conditionally on the event that there is a chain of excursions+loops joining the two vertical sides of the rectangle, the probability that $N$ is even is $1/2$ (and therefore the conditional probability that $N$ is odd is also $1/2$) and that furthermore, the conditional laws of the occupation field (of the union of the excursions and loops) when one further conditions on $N$ being odd or even are the same. So, in a way, this is the analogue of the switching property when one looks at $R$ and identifies the left boundary as one boundary point $x$ and the right boundary as one boundary point $y$, and one chooses one particular value for $a=b$. The analysis in \cite {LQW} proceeds as follows: the first part (which is an additional step that is specific to the continuous setting) is to argue (using partial explorations of loop-soups and conformal invariance) that this setup can indeed be viewed as the conformal image of the remaining-to-be-explored part of a partially discovered loop-soup, which is a non-trivial matter (other related work includes \cite {Q,QW2,MW}). This provides the motivation to look at this rectangular setup with the particular choice of boundary conditions. Then, the  analogue of Lemma 11, i.e., Theorem 2 in \cite {LSW} is derived via explicit computations for the law of the occupation fields in the same spirit as the proof that we just presented in this section. However, there are several features that make these explicit computations somewhat different than in our cable-graph case. An obvious first aspect is that the occupation time measures (and the square of the GFF) in the continuum have to be defined in a ``renormalized sense'' and that arguments based on the equivalence between sign-clusters and clusters of loops have to build on the existing literature about continuum two-dimensional loop-soups, such as \cite {L2,QW} that in turn build on \cite {LW,SW}.

Note also that in the continuous two-dimensional setting, a new feature (that was actually the main motivation for \cite {LSW}) appears: There is a subtle difference between the loop-soup clusters and their closure. The latter is sufficient to determine the occupation field, but it turns out (and this is the case in the odd-even switching) that some special exceptional points that are in a loop-soup cluster before the switching are not part of a loop-soup cluster anymore after the switching. This type of feature does not hold on cable-graphs. Indeed, if $x$ belongs to some Brownian loop in the loop-soup, then it is either in the interior of the trace of this loop, or it is one of the finitely many boundary points of the trace of this loop. But in the latter case, it is then almost surely in the interior of the trace of some other loops in the loop-soup. As a result, we see that almost surely, no point that is on the boundary of a loop-soup cluster (and there are almost surely only finitely many such boundary points for each loop-soup cluster) is actually on the trace of a Brownian loop in the loop-soup. So, the actual loop-soup clusters are a deterministic function of the occupation time field in the cable-graph case.

\subsection {The cable-graphs $\R_+$ and $[0,1]$}
\label{S63}
Let us make some remarks about the case where the cable-graph is $\R_+$. In this case, we in fact discussing squares of Bessel processes of integer dimensions in connection with Ray-Knight Theorem type identities, and the parity and switching properties become reinterpretations of some of Marc Yor's beloved identities in law for excursions, meanders etc. of the type described in \cite {Y} (and it is nice to see the probabilistic worlds of Aizenman and Pitman-Yor merging in this way!).
\begin {itemize}
 \item
One might start with the special limiting case of $\R_+$ with $x=0$ and $y = \infty$. In that context, the GFF is just Brownian motion, and for the $a=b=0$ case, we are conditioning $0$ and $\infty$ to be boundaries of the same GFF cluster, i.e. the Brownian motion/GFF to be positive on $(0,\infty)$. It is well-known that this intuition can be made rigorous and that the obtained process is a three-dimensional Bessel process.
On the other hand: The occupation time process of an (unconditioned) Brownian loop-soup in $\R_+$ is the square of a Brownian motion (since the Brownian motion is exactly the GFF in that case), and the occupation time intensity of the excursion from $0$ to $\infty$ (which is a three-dimensional Bessel process -- this one indexed by time!) is the square of a two-dimensional Bessel process when parameterized by space (there are several ways to see this -- one convoluted proof in the spirit of loop-soups is to consider the time-reversal of this process to be performing Wilson's algorithm in $\R_+$, and the collection of erased loops is then a loop-soup with twice the intensity of our ``standard'' loop-soup, see \cite {WP}). The sum of these two contributions is then indeed the square of a three-dimensional Bessel process.
\item
Similarly, when $x=0$, $a=b=0$ and $y$ is finite, one obtains the description of the Brownian excursion of time-length $y$ as a bridge of a three-dimensional Bessel process (from $0$ to $0$) which is David Williams' description of the excursion measure proved in \cite {R}.
\item The statement for non-negative $a$ and $b$'s provides some further ``Pitman-Williams-Yor'' identities in law involving Bessel bridges with randomly chosen dimensions in $\{ 3,7,11 \ldots\}$. More precisely, if we now consider the cable-graph $[z, \infty)$ (for some $z<0$), $x=0$ and $y=1$, the switching lemma states that for some explicit constant $C$ (that can be shown to equal to 1):
\begin {corollary}
\label {Cpy}
The law of a Brownian bridge $(X(t))$ on $[0,1]$ from $X(0)=a>0$ to $X(1) =b>0$ conditioned to be positive is identical to
the law of the square root of the sum of: (1) A squared Bessel process of dimension $0$ started from $a^2$ conditioned to hit $0$ before time $1$ (2) A time-reversed squared Bessel process of dimension $0$ started from $b^2$ at time $1$ and conditioned to hit $0$ on $(0,1)$, and (3) The square of a $(1+2\Delta)$-dimensional Bessel bridge from $0$ to $0$, where $\Delta$ is a Poisson random variable with mean $abC$ conditioned to be odd (i.e., one first samples this random variable $2\Delta+1$, and then the Bessel bridge of that dimension).
\end {corollary} The three contributions (1-3)  correspond respectively to the point process of excursions away from $0$ (that do not hit $1$), to the point process of excursions away from $1$ that do not hit $0$, and to the point process of excursions joining $0$ and $1$ that is conditioned to be odd (each excursion adds a two dimensions to the Bessel process). The parity lemma says on the other hand that the law of the square of the unconditioned bridge of reflected Brownian motion is obtained in the same way, but where the same Poisson random variable is not conditioned to be even (and can therefore be equal to $0$) -- giving rise to squared Bessel bridges of randomly chosen dimension in $\{1,5,9,\ldots\}$ instead.

One would think that Jim Pitman and Marc Yor would have come up instantaneously with proofs along the lines that we give in this section via explicit Laplace transform computations. As it turns out, in what probably is one of their famous joint summer papers \cite {PY}, they point out exactly such a decomposition of Bessel Bridges (of general dimension), by recognizing the different terms in the Laplace transform! In particular, Formula (1.f) in the case $d=1$ gives the parity lemma for the cable graph $[0,1]$. When one conditions a Brownian bridge not to hit the origin, one gets exactly a bridge of a three-dimensional Bessel process, so that Formula (1.f) in the case $d=3$ for the decomposition of such a Bessel bridge gives exactly the above corollary; for both cases, one just has to remember the formula for the Gamma function evaluated at half-integer points to read in the Formula (1.f) of \cite {PY} that the number of excursions joining $0$ and $1$ is a Poisson random variable respectively conditioned to be even or odd.
It is interesting and nice to see how Marc Yor's favourite techniques can be also reinterpreted via loop-soup decompositions. This is the type of ``quest for mysteriously hidden pathwise explanations'' of identities in law between functionals of stochastic processes that he was fond of.

As we will now see in Section \ref {Snew}, one avenue to proving the general switching property cwill be to use this case of the graph with one edge and ideas from \cite {LW3} to use ``equivalence between effective resitance'' (this is also related to the arguments in \cite {Aid1}); all these different routes are of course interrelated (since the Pitman-Yor computations bear similarities with the ones presented in the previous section).
\end {itemize}

\section {Switching along given circuits or paths}
\label {NewSection}
\label {Snew}

Let us now describe a third line of ideas, that allow to derive further statements such as Proposition \ref {Pnew} and its consequences.
One can view this either as a reworking/reformulation of parts of the previous proofs, or as a  third way to think about the switching property and derive it.  This approach,  which is arguably the one that gets actually closer to the actual ``switching'' name,  has also been independently worked out by Lorca Heeney before the posting of this version of this preprint. Here, the starting point is to note that when one first conditions $\Lambda$ on its values at all the sites, then one can relate the GFF to an Ising model and its FK-Ising model as explained in \cite {LW1} -- random even subgraphs also show up there --, which allows to use the tricks of the random current representation of the Ising model part.

In this section, we will use (and not reprove) the switching identity on single edges (i.e., Corollary \ref {Cpy}), which (as we have already mentioned) can be derived directly in several different ways (reminiscent of the two ways in which we derived the general switching identity so far, via Laplace transform/Dynkin isomorphism identities as in \cite {PY} or via loop-soups as in \cite {LW1,Aid1}). With this in hand, one will be able to build on features of the discrete world with its random current identities. As however noted in Remark \ref {edgetograph}, the way to go from one edge to the entire graph can be also understood via the effective resistance equivalence ideas from \cite {LW3}. The set of possible proofs of the switching identitites feels itself a little bit like the set of a paths between two points in a graph...

\subsection {Random even subgraphs and switching along loops inside a loop-cluster}

The main starting point is the observation, that comes directly from the rewiring property, that when one conditions on the value of $\Lambda$ at all the sites of the graph, then the traces of the parts of the loop-soup inside each of the edges is made of inputs (excursions away from the endpoints, and loops that stay in the edges) where what happens on different edges is independent of each other, up to the conditioning by a parity condition (at each site, the number of excursions that join this site to a neighbouring site has to be even) that affects only the excursions that cross edges.

Let us be more precise: One can chop the loop-soup ${\mathcal L}$ into its pieces/excursions ${\mathcal L}_e$ on each of the edges $e$. More precisely, for each $e$ that joins two sites $u$ and $v$ (we will keep this notation for the extremities of $e$ in the next few paragraphs),  ${\mathcal L}_e$ will consist of the collection $B(e)$ of excursions from $u$ to $u$ or from $v$ to $v$ in $e$ (i.e., the union of all excursions in $e$ away from the boundary of the edge that do not cross it), the collection $C(e)$ of crossing excursions that join $u$ and $v$ in $e$ (there will be only finitely many of those, we call their number $N(e)$), and there will also be a collection $L(e)$ of loops that are entirely contained in $e$. Note that (simply because this is created by a collection of loops), for each site $u$, the sum over all edges adjacent to $u$ of $N(e)$ (we count it twice if the edge is from $e$ to $e$) is necessarily even.

Let us first describe some considerations that where already present in \cite {LW1}:
We can start with the multiple-point version of Lemma \ref {parity1} when one chooses these points $x_1, \ldots, x_n, \ldots$ to be all the sites of the graph, i.e., the description of the conditional distribution of $({\mathcal L}_e)_{e \in E} := (B(e), C(e), L(e))_{e \in E}$ given the values of $\Lambda$ at all the sites of the graph. Unsuprisingly, $B$ and $L$ will be independent Poisson point processes of excursions and loops, while the collection of crossings will be a Poisson collection of crossings with a parity condition at each site.
More precisely, let us consider a collection of independent triples $(\tilde {\mathcal L}_e)_{e \in E} = (\tilde B(e), \tilde C(e), \tilde L(e))_{e \in E}$ where $\tilde B(e)$ is a Poisson point process with intensity
$a_u^2 \nu_{u,u,e} + a_v^2 \nu_{v,v, e}$, $\tilde C(e)$ is a Poisson point process of intensity $a_u a_v \nu_{u,v;e}$ and $\tilde L(e)$ is a loop-soup in $e$. Then:
\begin {lemma}[Proposition 7, \cite {W1}]
The conditional law of the collection $({\mathcal L}_e)_{e \in E}$ given $(\Lambda(x))_{x \in S} = (a_x^2)_{x \in S}$ is that of
$(\tilde {\mathcal L}_e)_{e \in E}$ conditioned on the parity event ${\mathcal P}$ that for all $x \in S$, the sum of $\tilde N(e)$ over all edges $e$ adjacent to $x$ is even.
\end {lemma}
As we have already pointed out after stating Lemma \ref {parity1}, the proof of Proposition 7 in \cite {W1} was left as a ``simple exercise''. The more detailed proof goes along exactly the same lines as that of Lemma \ref {parity1} in Section \ref {S3}, based on the formula for the law of the unconditioned number of crossings (Proposition 2.46 in \cite {WP}).
This result was then used in \cite {LW1} to relate this model to the Ising model and its random current representation, and to random even subgraphs.
For each edge $e$, we define $E_{\tilde \Lambda}$ to be the set of edges for which  $\tilde \Lambda_e > 0$ (where $\tilde \Lambda_e = (\tilde \Lambda_e (x), x \in e)$ is the sum of the local times of the three contributions parts of  $\tilde {\mathcal L}_e$).

One can now note that for each edge,
$$P [ \tilde N(e) \hbox{ is even} \ | \ e \in E_{\tilde \Lambda} ] = P [ \tilde N(e) \hbox { is odd}\ | \ e \in E_{\tilde \Lambda} ] = 1/2.$$
(this can be derived directly from the explicit value of the mean of the Poisson random variable $\tilde N(e)$ and the value of the probability that $\tilde \Lambda > 0$ on $e$ -- as we did in proof in Section \ref {S3}, see also \cite {LW1}; since we will use the switching property on each edge anyway, we can also view this also as a consequence of the switching along edges).
This immediately shows that the conditional  probability of the parity event ${\mathcal P}$ for ${\tilde L}$ given $E_{\tilde \Lambda}$ is just the probability that a uniformly chosen random subgraph of $(S, E_{\tilde \Lambda})$ is an even subgraph, which is a function of $E_{\tilde \Lambda}$ only (i.e., it does not depend on the chosen values for $\tilde \Lambda$ at the sites). Summing up (again, this is so far all very close to ideas present in \cite {LW1}):
\begin {lemma}
\label {lemmacond}
The conditional law of ${\mathcal L}$ given $E_{\Lambda}$ and $(\Lambda(x))_{x \in S} = (a_x^2)_{x \in S}$ can be described as follows: For each connected component of $(S,E_\Lambda)$ choose independently a uniform random even subgraph -- this gives a subset $E(P)$ of $E_\Lambda$. Then, independently for each edge $e$, choose ${\mathcal L}_e$ according to (a) the law of $\tilde {\mathcal L}_e$ conditioned on $\tilde \Lambda$ hitting $0$ if $e \notin E_{\Lambda}$, (b) the law of $\tilde {\mathcal L}_e$ conditioned on having an odd number of crossings if  $e \in E(P)$, and (c) the law of $\tilde {\mathcal L}_e$ conditioned on having an even number of crossings and $\tilde \Lambda_e >0$ in the remaining case where $e \in E_{\Lambda} \setminus E(P)$.
\end {lemma}

Now, this is where we can use the full version of the switching identity on each edge. Indeed, the choice of the random even subgraph of $E_{\Lambda}$ does not affect the conditional law of $\Lambda$ itself (since options (b) and (c) will provide the same conditional distribution for $\Lambda_e$). In other words, conditionally in $E_{\Lambda}$, the random even subgraph of $E_{\Lambda}$ that describes the parity of the number of jumps and the field $\Lambda$ are independent.
This is exactly Proposition \ref {Pnew}.

\medbreak
A more concrete reformulation of this result is the following:
\begin {proposition}
\label {Palternative}
Consider any finite self-avoiding cycle $c$ of edges in $G$. Then, there exists a measure-preserving involution ${\mathcal L} \leftrightarrow {\mathcal L}^*$ on the set of loop-soup realizations for which $c \subset E_{\Lambda}$ that keeps $\Lambda$ (and therefore $E_{\Lambda}$) unchanged and changes the parity of the number of jumps on the edges of $c$ (and on these edges only).
\end {proposition}

Since this switching does not change $\Lambda$, it is easy to see that it is possible to first choose $c$ to be a deterministic function of $E_\Lambda$, and then do the switching. One can also do the switching successively along $n$ cycles that are functions of $E_\Lambda$.

It is probably worth recalling some very basic topological features about even subgraphs of a given graph and how they relate to elementary topological features: If one has a finite connected graph, then one can choose a maximal subtree of this graph, and then consider the number $h$ of edges that are missing (i.e., the edges $E$ in the graph that are not in the tree $T$). It is easy to see that on the one hand, this number is the number of generators of the fundamental group of the graph, and on the other hand that every function from $E \setminus T \to \{ 0, 1\}$ can be extended in a unique way into a function of $E \to \{ 0, 1 \}$, so that $\{ e, \ \xi (e)  = 1 \}$ forms an even subgraph of $E$ (always with the same set of sites).
In particular, we see that the cardinality of the set of even sugraphs of $E$ is $2^h$ (of course, the number $h$ can also be described in many other ways, but we stick to this down-to-earth approach here). Furthermore, if $c_j$  is the (unique) simple cycle in the graph that is obtained by adding the $j$-th missing edge to $T$, then one has a one-to-one correspondance between even subgraphs and $\{ 0, 1 \}^h$ by choosing edge to be in the subgraph if it belongs to an odd number of cycles for which $\eps_j =1$ (when $(\eps_j) \in \{ 0,1 \}^h$). In other words,  one starts with the empty graph and then successively switches the parity along each $c_j$ for which $\eps_j =1$.
A uniform even subgraph can therefore be obtained by tossing $h$ independent coins to choose $(\eps_j)$.

\subsection {Switching with more insertions}

The same ideas can be used to obtain yet another simple proof for Theorem \ref {mainthm}. In fact, one can obtain in the same way the following generalization of Theorem \ref {mainthm}:
\begin {proposition}
For any $2n$ points $x_1, \ldots, x_{2n}$ and any positive numbers $a_1, \ldots, a_{2n}$, the conditional law of $\Lambda$ given $\Lambda(x_1)=a_1^2, \ldots , \Lambda(x_{2n}) = a_{2n}^2$ and the event  $A(x_1, \ldots, x_{2n})$ that for each cluster of loops $C$, $\# ( C \cap \{ x_1, \ldots, x_{2n} \}) $ is even,
can equivalently be obtained in the following manner: Consider the total occupation time of the overlay of the following three inputs: (1) A Brownian loop-soup in $G \setminus \{ x_1, \ldots, x_{2n} \}$, (2) Independent Poisson point processes of Brownian excursions with intensity $a_i^2 \nu_{x_i, x_i; G \setminus \{ x_1, \ldots, x_{2n}\}}$ (for each $i$), (3) Independent Poisson point processes of Brownian excursions with intensity $a_i a_j \nu_{x_i, x_j : G \setminus \{ x_1, \ldots, x_{2n} \}}$ for $i < j$, but that are then conditioned on having a total number of excursions away from $x_i$ to be odd, for each $i$ (so if ${\mathcal N}_{i,j}$ is the number of excursions joining $x_i$ and $x_j$, then $\sum_{j  \in \{ 1, \ldots 2n \} \setminus \{i \}  } {\mathcal N}_{i,j}$ has to be odd for each $i$).
\end {proposition}

So, the situation (and this is of course no surprise in the light of \cite {LW1}) is a bit similar to what the moment identities for the Ising model tell about the connectivities in the underlying FK model. When $2n=2$, this describes exactly the law conditioned on $x_1 \leftrightarrow x_2$, but when $2n \ge 4$, it only describes the conditional law given part of the connectivity features between the $2n$ points.

\medbreak

One simple way to derive these statements with ``sinks'' from the Proposition \ref {Pnew} or Proposition~\ref {Palternative} (where no sinks where presents) is to use the classical trick of ``adding extra edges'' to the graph and to use the spatial Markov property of the model. Let us briefly outline how this works in the case of two points (i.e., how to deduce Theorem \ref {mainthm} from Proposition \ref {Pnew}): We simply add an edge $e$ joining the two points $x$ and $y$ (the edge can be of any given length) to the cable graph $G$, which now gives a new cable graph $G'$.
We can now apply Proposition \ref {Pnew} in $G'$. On the set of loop-soup configurations where $x$ is connected to $y$ within $e$ as well as in $G$ (which corresponds to the existence of a self-avoiding cycle of edges in $E_{\Lambda'}$ that goes through the edge $e$ for $G'$), we therefore have a measure-preserving  bijection that also preserves $\Lambda'$ between the configurations where the number of crossings of $e$ is even (which when restricted to $G$ corresponds exactly to the loop-soup configurations in $G$ conditioned on $x \leftrightarrow y$) and the configurations
where the number of jumps along $e$ is odd (which when restricted to $G$ corresponds to the configurations with the odd number of excursions between $x$ and $y$ as in the proposition).

For the general case of $2n$ points, one can similarly add  $2n$ edges connecting each of these $2n$ sites to one additional point, and then (for the configurations in the large cable graph $G'$) switch from the event where $A(x_1, \ldots, x_{2n})$ holds and all these additional edges have an even number of crossing (that corresponds to the usual loop-soup in $G$) to the event where $A(x_1, \ldots, x_{2n})$ holds and all these additional edges have an odd number of crossings (that corresponds to the description with the odd condition in the Proposition) by switching along $n$ loops in $G'$ (that are well-chosen depending on the connections between the $2n$ points in $G$) that go through all these $2n$ additional edges (i.e., one loop for each edge).

\subsection {Windings of clusters, windings of loops, Lupu's intensity doubling conjecture}

We now turn to considerations related to Corollary \ref {Cnew}. Recall the definition of the loop-soup parity around $\Delta$ of a loop-soup cluster from the introduction.
If we condition on $\Lambda$, then for each cluster that contains a cycle $\gamma$ with odd index around $\Delta$, we can use the switching that swaps the
parity of the number of crossings of each of the edges of $\gamma$. Elementary considerations then imply that this switching changes the parity of $P_\Delta(C, {\mathcal L})$, so that the corollary immediately follows.

\begin {remark}
Corollary \ref {Cnew} can be generalized in several ways:
\begin {itemize}
 \item If a cluster only contains cycles of index in $n\Z$ for some $n \ge 2$ around $\Delta$ (and does indeed contain such cycles), then one can define the ``mod-$n$ parity'' $P_{n, \Delta}$ of the loop-cluster to be $0$ or $1$ depending on whether the total index of the loop-soup is in $2n \Z$ or not. Then, the very same arguments (switching along one cycle of index $n$ in the cluster) shows that $P_{n, \Delta}$ is equal to $0$ or $1$ with probability $1/2$.
 \item We can use windings around other  $(d-2)$-dimensional hypersurfaces than the $(d-2)$-dimensional hyperplanes. For instance, in $\Z^3$, one can look at the winding number around any closed loop that does not intersect the cable-graph of $\Z^3$ -- the analog in any dimension works. The same results will clearly hold (just using the switching along any cycle that wraps around the loop or the hypersurface).
 \item For each cluster $C$, one can also look at parities around different $\Delta$'s at the same time (up to $h=h(C)$ of them). Indeed, one can choose $h$ ''independent'' cycles in $C$, and because of Proposition \ref {Pnew}, we see that the joint conditional law of the parities of indexes of the loop-soup around these $\Delta$'s will be that of independent Bernoulli variables.
 \end {itemize}
\end {remark}

When the loop-soup parity (or the mod-$n$ parity) of a cluster (around some $\Delta$) is $1$, then it implies that for at least one loop in the loop-soup, the index is not zero (of course, the converse is not true). In particular, this shows that for any cluster that contains a cycle with non-zero index around $\Delta$, the conditional probability that it does actually contain a Brownian loop with non-zero index is at least $1/2$ (even if one conditions on $\Lambda$), which is not such an intuitive fact.

\medbreak

As we will explain in more detail in \cite {LuW}, this allows to quickly prove the intensity doubling conjecture of Lupu \cite {L3} for the cable-graphs $\Z^d$ in dimensions $d \ge 7$. The basic ideas are the following:
\begin {itemize}
 \item
When $d \ge 7$, it is easy so see that if we look at the loop-soup in $[-N,N]^d$, then for any $a>0$,  the probability that there exist two different Brownian loops of diameter greater than $aN$ that are in the same cluster goes to $0$ as $N \to \infty$ (simply because the expected number of pairs of sites $(x,y)$ in the graph such that $x$ is in a big loops, $y$ is another big loop, and $x$ is connected to $y$ via the union of all the other Brownian loops goes to $0$ as $N \to \infty$ (the two-point function implies immediately gives an upper bound $O(1/N^{d-6})$).
\item
Suppose that we observe a loop-soup cluster that contains a large cycle that wraps around some (large) $N\Delta$, and that the whole cluster remains at distance greater that $a N$ from $N \Delta$. Then, the parity-index (or $n$-parity-index) of the loop-soup in this cluster will be $0$ or $1$ with probability $1/2$. But the previous step then implies that with probability is $1/2 + o(1)$ (when $N \to \infty$) the cluster does in fact contains no large Brownian loop wrapping around $\Delta$ and the index-parity is $0$, and with probability $1/2 + o(1)$ it does contain exactly one Brownian loop that wraps around $\Delta$.
\end {itemize}
So, provided one can show that any cluster that contains a large Brownian loop will actually be of latter type (for some $\Delta$ and $a$) with high probability, we conclude that the collection of clusters in $\Lambda_N$ that do contain large cycles have the following property in the $N \to \infty$ limit: If one selects each of them with probability $1/2$ independently, one ends up with the collection of clusters that do contain the Brownian loops in the loop-soup. This in particular also shows that the law of the discarded ones (that are also selected with probability $1/2$ each) is the same. So, the collection of large cycles that do contain no big Brownian loop will give rise to a second independent loop-soup in the scaling limit! This can be compared (and was one motivation for) the results presented in \cite {CW} about usual critical Bernoulli percolation in high-dimensions.

We also refer to \cite {LuW} for other considerations related to the twisted GFF, and the relation with the results of \cite {L3} and \cite {KL}.

\subsection {Outlook: An explicit bijection via ``peeling''}
 \label {Speeling}
 \label {Se5}

 In the previous sections, we discussed an explicit way to perform the switching on a graph, once one knows how to proceed at the level of the individual edges. This does however not yet provide a concrete recipy explaining how to do the switching from even to odd number of crossings inside each of the edges where one wants to do this.

 We will now outline ideas that lead to such explicit bijections at the actual level of Brownian loops. This section is describing at a heuristic level ideas that will be developed in the forthcoming paper \cite {LuW}. Since there is no real difference in explaining this for general cable graphs compared to the intervals, we do present this in the former case (but the reader can also assume that our cable graph is just a segment). Again, just as in the previous section, a number of the underlying ideas were briefly mentioned in \cite {LW1}.

 One first remark is that the fact that in Dynkin's isomorphism, the conditional law of the remaining squared GFF once one has removed the contributions of the excursions away from $x$ is just that of a squared GFF in $G \setminus \{ x\}$ indicates that the union of all these excursions (away from $x$) is in fact a local set of the GFF that one starts with (in the Schramm-Sheffield sense \cite {SchSh, WP}). It appears quite natural to start from $\Lambda$ and to try to find some ``exploration-type'' algorithm to construct a coupling between the conditioned GFF that one starts with and the set of loops that go through $x$ (so we are here first more in the setup of Dynkin's isomorphism).

 This question is reminiscent of papers by Warren, Yor, Aldous and others \cite {Wa,WY,A} that were studying the conditional distribution of one Brownian motion (or excursion) up to some given (stopping) time $T$ given its local time profile at $T$ -- one idea being to try to progressively draw part of the Brownian motion, that then eats up progressively the allotted local time, and to describe the joint dynamics of the pair (Brownian motion, remaining allotted local time). Here, the question is how to do something similar for collections of Brownian loops on graphs. This goes (initially) very much along the same lines as \cite {AHS}, see also \cite {Aid1}.

 Again, it is useful to make a few remarks in the discrete-time setting. Suppose that one has sampled a loop-soup and knows the number of jumps $n(e)$ along each edge $e$ (and therefore the number of visits $n(y)$ of each site $y$) -- one counts here jumps along unoriented edges, so that there will be $2n(y)$ available edge-jumps at each site. One way to think of it is to draw $n(e)$ different parallel edges instead of each edge $e$. Then, the law of the decomposition into loops given this number of jumps along all edges is obtained by pairing uniformly at random  the $2n(y)$ incoming possible edges  independently for each site $y$ -- this is the rewiring property explained in \cite {W1}.

 Let us now fix the site $x$ and try to explore all the loops that go through $x$.  One starts at $x$ and chooses one of the $2n(x)$ adjacent edges uniformly at random and follows it. One therefore moves along an edge $e$ with a probability $n(e) / (2n (x))$. Then, at the next site $x_1$, the edge that one arrived with will be paired with one of the $2n(x_1)-1$ remaining ones at this site. One then continues like this until one eventually exhausts all the possible edges away from $x$. We see that in this way, one simply follows an excursion away from $x$ by one of the loops. When this walker returns to $x$, one may discover that one it paired with the initial first edge (in which case one has closed one of the random walk loops and then proceeds by choosing another one), or one is paired with another edge away from $x$ and one continues.
It is elementary to see that the number of visits of $x$ by the first discovered loop will be uniform on  $\{1, \ldots, n(x)\}$.

 In the ``fine-mesh limit'' (i.e. for instance when one subdivides each edge into $K$ edges) where the number of visits to a site tends to infinity, one can note that the first loop that one discovers will be a ``macroscopic one''.
 Indeed, in this limit, the local time at $x$ of the first encountered loop will be $X_1 a^2$ (where $X_1$ is a uniform random variable in $[0,1]$), the local time at $x$ of the second loop will be $X_2 (1-X_1) a^2$ and so on -- this stick-breaking description is of course classically related to cycles in random permutations.

 The obtained trajectory is then the ``reversed vertex reinforced jump-process'' as introduced by Sabot and Tarr\`es in \cite {ST} and further studied in \cite {LST}. In the limit where one lets $K \to \infty$, it is explained by Lupu, Sabot and Tarr\`es \cite {LST2} that the process converges to a particular continuous self-interacting process. As opposed to the more strongly ``locally self-repelling'' processes such as the one defined and studied in \cite {TW}, the sample trajectories of this process will look very much like Brownian ones (indeed, the ``annealed version'' will just give rise to the Brownian excursions away from $x$ by the Brownian loops, so that they always will have a finite quadratic variation), but it can be viewed as a Brownian motion with a rather complicated drift that is determined by its ``remaining local time function'' (the initial local time profile minus the local time of the past of the process). This updated remain profile  will decrease when the process goes along.
 This process is well-defined and the convergence to this continuous process holds true up to the possible first time at which the local time at the current position of the process hits $0$. Note that after that time, the process can not come back to this point anymore (as it would otherwise create a positive local time at this point, which is not possible), so that it has ``to decide'' in which direction to continue. Since this decision is ``locally made'' based on the local time profile, it is quite natural to  guess that it will choose one of the two directions with equal probability.

 Another direct way to describe this process is of course just to start with a Brownian loop-soup on the cable graph, to consider all the Brownian loops that go through $x$, to order them weighted by their respective local time at $x$ and to concatenate them.
 A first remark is that for each edge adjacent to $x$, it will then necessarily happen that at some time during this exploration, the process will be on that edge and the remaining local time at the current position of this process hits $0$ (and this will happens before having discovered all the excursions away from $x$).  Indeed, if this was not the case, then since the process would exhaust all the excursions away from $x$, it will at least have to reveal some of the zeroes of the squared GFF with $0$ boundary condition at $x$ obtained by removing all the excursions away from $x$.  In particular, this shows that there will almost surely be a first time $T$ at which the local time profile at $T$ disconnects $x$ from $y$ (since $T$ can not be larger than the last time at which the process has created a zero on each of the edges adjacent to $x$).
 At that time, the process (as defined above) will not be able to choose one of the two available options. Indeed, given that it eventually has to come back to $x$, it has to choose to go ``towards $x$''.
 On the other hand, if we were to force it to go ``towards'' $y$, we would end up adding one single excursion from $x$ to $y$ instead.

 More precisely, suppose that at this disconnection time $T$, the process has done $k$ excursions between $x$ and $y$ (so for instance, if $k=3$, it has visited $y$, then $x$, then $y$ but not been back to $x$ yet). Then, we see that depending on whether the process goes back to $x$ or to $y$, one ends up with a total of $k$ or $k+1$ excursions. Furthermore,  when the process goes back to $x$, then the total number of excursions is even, and the process is indeed the discovery of the loops, while if the process is forced to go towards $y$, it discovers an odd number of excursions. So, the final forcing towards $y$ at the disconnection time changes the parity of the number of excursions, and it is indeed one natural candidate for the explicit bijection.

 In particular, this will show that the parity switching of the loop-soup within each single edge (that is then used along loops in Proposition \ref {Pnew} and along paths for Theorem \ref {mainthm}) can be realized in such a way that the actual number of crossing of the edge changes by $\pm 1$ (and this still does not affect the total occupation time).

Again, we refer to the forthcoming work \cite {LuW} for more details on this approach and its consequences.

\section {Immediate consequences for incipient infinite clusters, interlacements and related questions}
\label {S2}

We now discuss some of the  direct consequences/reformulations of the switching properties that heuristically correspond to their version when one of the two points $x$ or $y$ is sent to infinity. So, this will be about the infinite infinite cluster measure (as already stated in Theorem \ref {thmIIC}) and the combination of loop-soups with random interlacements.
We will address  some related points also in Section \ref {S6}.

In all of this section, we will again consider excursions between different points ``with their end-points removed'' (just to ensure the exact identity between the different descriptions of clusters in the case where one of the marked points is conditioned to be a boundary points of a cluster).
We will first discuss results in $\Z^d$ and then discuss the case of more general graphs. Finally, we will make some comments on the massive case.

\subsection {Incipient infinite cluster in $\Z^d$}
\label {S2.1}

Let us first complement Theorem \ref {thmIIC} with the corresponding statements for
conditioned versions of the IIC  in $\Z^d$ for $d \ge 3$:
\begin {proposition}[Conditioned versions of the IIC measure]
\label {propIIC}
\begin {enumerate}
 \item The limit when $y \to \infty$ of the law of the square of the cable-graph GFF conditioned on $x$ being a boundary point of the cluster containing $y$ exists. The law of this ``incipient infinite cluster measure conditioned on $x$ being a boundary point of the incipient cluster'' is simply obtained via the occupation time of the overlay of a Brownian loop-soup in $\Z^d \setminus \{x \}$ with one independent Brownian excursion from $x$ to infinity.
 \item The limit when $y \to \infty$ of the law of the square of the cable-graph GFF conditioned on $\Lambda (x) = a^2$ and $x \leftrightarrow y$ exists. The law of this incipient infinite cluster measure conditioned on having an occupation time $a^2$ at $x$, is described by the occupation time of the overlay of a Brownian loop-soup in $\Z^d$ conditioned to have occupation time $a^2$ at $x$  with one independent Brownian excursion from $x$ to infinity.
\end {enumerate}
\end {proposition}

Note that in all these results, one can obtain the corresponding laws of the conditioned GFF itself, by simply choosing a sign independently for each of the obtained sign clusters.

\medbreak

 Constructing incipient infinite clusters for critical percolation models has a long history, starting (on the mathematical side) with Kesten's \cite {K}. It is striking that for this Brownian loop-percolation model,  one can get this construction in such an effortless way, as well as a nice simple description of the IIC (somewhat reminiscent of the IIC in a tree, with a collection of critical trees attached to a backbone).

\medbreak
Proposition \ref {propIIC} and Theorem \ref {thmIIC}
 follow directly from the switching property and the fact that the law of the excursion from $x$ to $y$ converges as $y \to \infty$  to the law of the excursion from $x$ to infinity (in the sense that the law of the former up to its last exit time of the ball of radius $K$ around $x$  converges to the law of the latter up to its last exit time of this ball), which is a classical fact in $\Z^d$. Note that the reweighing in Theorem \ref {thmIIC} is easy to define (the unweighted law  of the local time at the origin being a squared Gaussian random variable). It is also easy to check that the limiting measures in Proposition \ref {propIIC} can be viewed as conditioned versions of the one described in Theorem~\ref {IIC}.

\begin {remark}
Let us comment on recent closely related works: As we briefly mentioned in the introduction, the existence of the incipient infinite cluster measure for loop-soup occupation times in $\Z^d$ (for  all $d \ge 3$ except for $d=6$) has in fact been very recently obtained by Cai and Ding in \cite {CD3} by a rather less direct route (with tricky and astute dimension-dependent estimates and renormalization arguments that caused their proofs not to cover the case $d=6$). So, the switching property provides an alternative direct proof of the existence of the measure that works also for $d=6$ (and as we shall see, for a large class of transient graphs) and it gives also a new explicit very simple and workable description of the limiting measure.

It should  be stressed that \cite {CD3} contains also other characterizations of this IIC measure that do not follow so directly from the switching property.  It is for instance also shown in \cite {CD3} that the limiting procedure ``conditioning on $0 \leftrightarrow x$ and letting $x \to \infty$'' (which is the one that of Theorem~\ref {IIC} and is the most naturally related to the switching property) has the same limit as ``conditioning the cluster containing $0$ not to be contained in the ball of radius $n$ and letting $n \to \infty$''. See also aspects of the discussion in  Section \ref {Se6}.
\end {remark}

\begin {remark}
With this explicit description of the incipient infinite cluster in hand, it is  possible to revisit aspects of the proofs from \cite {GN} about the Alexander-Orbach conjecture in high dimensions (i.e., how does ``Brownian motion on the incipient infinite cluster behave?''). Note that the statements in \cite {GN} are about any ``subsequential IIC measure'' since at the time, the incipient infinite cluster measure was not known to be unique). We refer to the upcoming paper \cite {CW} for further details on such aspects in high dimensions.
\end {remark}

\subsection {Interlacements and loop-soups}
\label {S2.2}

The cable-graph interlacement in $\Z^d$ for $d \ge 3$ can be viewed as a Poissonian collection of Brownian excursions away from the ``point at infinity'' -- see \cite {Sz2}. The study of the geometry of the interlacements (and their complement in $\Z^d$) initiated by Sznitman has been a very active research topic in recent years.

One way to define the normalization of the intensity of the interlacement measure ${\mathcal I}$ (which is an infinite measure on bi-infinite Brownian paths that come from and go to infinity) is to say that the for Poisson point process with intensity  $\alpha^2 {\mathcal I}$ (we will call this an $\alpha$-interlacement), the mean occupation time of every point $x$ on edges of the cable graph is $\alpha^2$. Another choice of the normalization would then only change the values of the constant $C_0$ appearing in the proposition.

Consider such an $\alpha$-interlacement and an independent (critical) Brownian loop-soup on the cable-graph of $\Z^d$ and define $I=I(\alpha)$ to be the union of the two. Then, the analogue of Dynkin's isomorphism, known as Sznitman's isomorphism \cite {Sz1} states that the occupation times $\Lambda_{I(\alpha)}$ is distributed as $(\alpha + \Gamma)^2$ (where $\Gamma$ is a GFF in $\Z^d$). In other words, $\alpha$ plays the role of the ``value at infinity of the GFF''.

The following result should therefore  be hardly surprising, as it is a carbon copy of our main switching theorem, where one replaces the second point $y$ by infinity (the excursions away from $y$ become an interlacement, the excursions joining $x$ and $y$ become excursions from $x$ to infinity). The constant $C_0$ appearing in the statement can be easily determined explicitly from the Green's functions in $\Z^d$.

\begin {proposition}
\label {interlacementresult}
Conditionally on $x \leftrightarrow \infty$ in $I(\alpha)$ and on $\Lambda_{I(\alpha)}(x) = a^2  >0 $, the field $\Lambda_{I(\alpha)}$ is distributed as the occupation times of the overlay of:
\begin {itemize}
 \item A loop-soup in $\Z^d \setminus \{ x \}$.
 \item A Poisson point process of excursions away from $x$ with intensity $a^2\nu_{x,x, \Z^d}$.
 \item An $\alpha$-interlacement  conditioned not to hit $x$.
 \item A Poisson point process of excursions from
 $x$ to infinity with intensity $a \alpha C_0$ times the excursion probability measure from $0$ to $\infty$, which is {\sl conditioned to have an odd number of such excursions}.
\end {itemize}
\end {proposition}

So in particular, regardless of $a$, the sign-cluster containing $x$ does contain a Brownian excursion from $x$ to infinity. Note that the union of the first two items give the occupation time of a loop-soup conditioned on its occupation time at $x$.

The even/odd switching mystery is again quite apparent in this particular incarnation of the switching property. Indeed, in an interlacement, the Brownian paths ``to infinity'' come in pairs (one for each ``side'' of the bi-infinite Brownian paths from the interlacement), while the above description (once conditioned by $x \leftrightarrow \infty$) ends up with an odd number of paths to infinity.

\begin {proof}
One very natural way to approximate the infinite lattice $\Z^d$ is to consider the cable-graph $G_N$ defined as follows: It is the box $[-N,N]^d$ with one additional point $\partial$  ``at infinity''. To all the edges (with unit length) of $\Lambda_N$, one adds edges between all pairs of points on the boundary $\lambda_N$ of $\Lambda_N$ as well as edges between the points of $\lambda_N$ and $\partial$, where the ``length'' of these edges is chosen to fit the transition probabilities for the Brownian excursions in the complement of $\Lambda_N \cup \{ \partial \}$. Then, the trace inside of $\Lambda_N$ of the loop-soup and the interlacements (in $\Z^d$) correspond exactly to the trace inside $\Lambda_N$ of the loop-soup and the excursions away from $\partial$ in $G_N$. In particular, we can note that in the large $N$ limit, the event that $0$ is in an infinite component of $I(\alpha)$ will correspond (up to an event of probability that goes to $0$ as $N \to \infty$) the event that $0$ is connected to $\partial$ for the loop-soup and excursions in $G_N$.

One then just applies Theorem \ref {mainthm} (or rather its version on compact graphs, see Remark \ref {compact}) to the cable-graph $G_N$ with $y=\partial$ and $b=b(N,\alpha)$ and lets $N \to \infty$, which readily leads to the above proposition.
\end {proof}

This proposition has also an ``integrated version'' where one does not condition on $\Lambda(x)$ just as in the case of two points $x$ and $y$. We leave the details to the reader.

\medbreak

One can also consider the two following limiting cases:

\begin {itemize}
 \item [(a)]
In the limit when $\alpha \to 0$, one recovers the IIC measure (i.e., the overlay of a loop-soup with one excursion to infinity). In other words (but this comes here more as a consequence of the explicit description of both the IIC measure and the conditioned ``interlacement + loop-soup'' measure): The IIC measure is the limit as $\alpha \to 0+$ of the measure of the overlay of an interlacement with intensity $\alpha$ with an independent loop-soup, conditioned on $0$ being connected to infinity (i.e. to one of the excursions away from infinity).

\item [(b)]
In the limit when $a \to 0$, one gets the following result valid for all positive $\alpha$:
\begin {proposition}
 Conditionally on the event that $x$ is on {\em the boundary of} the unbounded connected component of the union of the $\alpha$-interlacement with a loop-soup, the law of the occupation times is that of the sum of the occupation times of an $\alpha$-interlacement conditioned not to visit $x$, an independent loop-soup conditioned not to intersect $x$ and one independent Brownian excursion from $x$ to $\infty$.
\end {proposition}
\end {itemize}
These results can be related to aspects of recent nice work on various ``near-critical'' cable-graph loop-soup models, see for instance Drewitz, Pr\'evost and Rodriguez \cite {DPR1}.

\subsection {IIC measures on other graphs}
\label {SLiouville}
We see that in the previous results on the IIC, the only feature of $\Z^d$ that we used was that the excursion probability measure from $x$ to $y$ was converging as $y \to \infty$ to some probability measure on paths from $x$ to infinity (that we then called the measure on excursions from $x$ to infinity).
This shows that it is possible to generalize the IIC measure result to general transient infinite graphs (possibly with a boundary) as follows:

\begin {proposition}[IIC for general graphs]
Suppose that $x$ is fixed and consider a sequence $y_n$ that goes to infinity, with the property that the probability measure on excursions from $x$ to $y_n$ converges as $n \to \infty$ to a probability measure on paths ${\mathcal P}$ from $x$ to infinity (in the sense that for all $K$, their law up to the last exit of the metric ball of radius $K$ around $x$ converges weakly).
Then, the law of the loop-soup occupation times (for a loop-soup on $G$) conditioned on $x \leftrightarrow y_n$ will converge (in the same sense as in Theorem \ref {thmIIC}) to the occupation of the overlay of a loop-soup with distribution reweighed by $\sqrt {\Lambda (x)}$ with an independent path chosen according to ${\mathcal P}$.
\end {proposition}

The similar statements for the other conditionings hold as well. This can then be used in different ways:
\begin {itemize}
 \item If the law ${\mathcal P}$ does not depend on the choice of the sequence $y_n$ that goes to infinity (this is for instance the case in $\Z^d$), we get a unique IIC measure.
 \item If the law ${\mathcal P}$ does depend on the choice of the sequence $y_n$ (this is of course not unrelated to the notions of Poisson boundary or the Liouville property), then one gets a family of different IIC measures. For instance, if the cable-graph is a transitive hyperbolic graph (where loops will actually typically be quite small), one ends up with infinitely many IIC measures (one with each point on the boundary) -- and this leads to a picture a bit closer to the IIC Bernoulli percolation measures on trees.
\end {itemize}

Similar ideas can be also used in the cases of interlacements+loop-soups.

\subsection {Comments on massive loop-soups}
\label {Smassive}

Recall the observation in the introduction that the switching identities work also for Brownian motion with killing rates. This includes the case of constant killing rates, i.e., the massive loop-soups where loops $\gamma$ from an ordinary Brownian loop-soup on the cable-graph are removed independently with a probability $1- \exp ( - m  T(\gamma))$ (where $T(\gamma)$ denotes the time-length of the loop). The corresponding occupation time $\Lambda$ is then the square of a centered Gaussian Field, but with correlation function given by the massive Green's Function $G_m$ instead of $G$. The last object involved in the switching property is then the massive excursion measure on excursions $\eta$ between two points $x$ and $y$, which is again just the usual Brownian excursion measure weighed by $\exp (- m T(\eta))$.

In many respects, the massive loop-soup in $\Z^d$ is to the normal loop-soup what a near-critical model is to a critical model. It looks the same at small scale but behaves in a subcritical way at large scale. The switching property provides for free features that are sometimes quite hard and technically involved when one studies subcritical models. For instance, we see that if we fix $m$ and condition on the origin being connected to $y_n$ (and then let $y_n \to \infty$ as $n \to \infty$), the (typically unique when $n \to \infty$) Brownian excursion appearing in the switching will be chosen according to the massive Brownian excursion measure, reweighed to be a probability measure. It is well-known how these excursion trajectories converge to straight lines when $n \to \infty$ with good control of the transversal fluctuations.
One can also use standard large/moderate deviations results when $m=m(n) \to 0$ sufficiently slowly. So, one not only has the description of clusters containing two far-away points in the usual critical loop-soup, but also the description of this entire window of near-critical regimes that come from the massive cases (for instance, one can see precisely how isotropy is lost simply by looking at the behaviour of the massive excursion measure), which is again something that is typically very hard to obtain for other percolation models.

\section {Some first consequences for $n$-point functions and other features}
\label {Se6}
 \label {S6}

As already illustrated in Section \ref {S2}, this switching property provides a useful additional tool to study various aspects of the large scale properties of these loop-soup clusters on various cable-graphs. It seems possible  to revisit aspects of the recent  papers on the topic, and to derive further results.  The goal of this section is to describe the type of considerations that become available using the switching property. More detailed and further results will be discussed in forthcoming work.

One general remark is that up to the present paper, the study of large sign clusters in $\Z^d$ was focusing first on the derivation of connection probabilities between points, or between points and sets in the case of one-arm probabilities. For instance, a large part of the recent papers \cite {DPR1,DPR2,DPR3,DW,CD1,CD2,CD3,GJ} is devoted  on establishing bounds for the  ``one-arm'' estimates, i.e. for the behaviour of the probability of the event $A_N$ that the sign cluster containing a given point is not contained in the ball of radius $N$ around this point when $N \to \infty$). Those proofs were mostly building on Lupu's two-point function (and in some case on an extension from \cite {LW3}) on the one hand, and on
insightful delicate explorations of the GFF, involving very non-trivial renormalization arguments.
With the multiple-point function estimates in hand (that they then subsequently obtain), it is then possible (see in particular very recent papers \cite {DPR3,CD3,GJ}) to extract information about the large-scale properties
of ``typical'' large clusters. This last step (from $n$-point functions to the properties of clusters) is not really specific to this loop-percolation case (see for instance the considerations of \cite {Ai2} for Bernoulli percolation).

In this section, we will only outline possible ways to obtain some of these estimates for multiple-point functions using the switching property.
But we would like to first note that
the switching property does in fact contain in itself more information  (that is then lost when one only keeps the multiple point estimates) because one has the simple description of the conditional distribution of the loop-soup percolation picture given the existence of the connection.
So, for instance, instead of obtaining a description of the large clusters in high dimension that would in some sense converge to integrated super-Brownian excursions in the sense that the uniform measure on the cluster does converge to the natural measure on the ISE tree (which is what the $n$-point estimates ultimately do), one can get a stronger description of the large clusters in terms of coupling with discrete trees on the cable-graph. Another upshot is that one will be able
to discard (in many cases) the existence of other large but skinny clusters (that would not contribute to the multiple-point estimates), a type of result that is out-of-reach via $n$-point estimates. We plan to discuss such new results in \cite {W4,W5,CW}.

\subsection {High dimensions}

As conjectured in \cite {W2} and confirmed in various aspects in the aforementioned papers by Cai and Ding, and by Drewitz, Pr\'evost and Rodriguez, loop-soup percolation in high dimensions (i.e,, $d> 6$) behaves in many ways in the same manner as Bernoulli percolation does (or is conjectured to be -- see \cite {Ai2}).  We will here only show how to derive some of the intermediate statements in these papers using the switching properties.

Consider loop-soup percolation in the cable-graph of $\Z^d$ intersected with the box $\Lambda_N = [-N,N]^d$. Let us start with choosing three points $x$, $y$ and $z$ that are all at distance greater than $cN$ from each other, and try to derive an upper bound for $P[ x \leftrightarrow y \leftrightarrow z]$. If we use the switching property, we immediately see that, when first conditioning on $ x \leftrightarrow y$, we then have to bound the conditional probability that at least one of the three scenarios occurs, when $A$ and $B$ are (correlated) Gaussian random variables:
\begin {itemize}
 \item That for a Poisson point process ${\mathcal P}_1$ of intensity $A^2$ of excursions away from $x$, an independent loop-soup connects $z$ to one of the excursions.
 \item That for a Poisson point process ${\mathcal P}_2$ of intensity $B^2$ of excursions away from $y$, an independent loop-soup connects $z$ to one of the excursions.
 \item That for a Poisson point process ${\mathcal P}_3$ of intensity $|AB|$ of excursions from $x$ to $y$ conditioned to be odd, an independent loop-soup connects $z$ to one of the excursions.
\end {itemize}
It is very easy for each lattice point $w$ to bound the probability that it is in ${\mathcal P} := \cup_{i=1,2,3} {\mathcal P}_i$ i.e., it will be bounded by some constant over $d(w, \{ x,y \})^{d-2}$. We can therefore use a brutal union bound (recalling also that $P [ x \leftrightarrow y ] \le C / |x-y|^{d-2}$ and that for the remaining loop-soup in $\Z^d \setminus \{ x,y \}$, the probability of $z$ being connected to $w$ is also bounded by $C / |w-z|^{d-2}$),
$$
P [ x \leftrightarrow y \leftrightarrow z ]
\le \frac {C}{|x-y|^{d-2}}  \times \sum_{w \in \Z^d} \left( P [d ( w,\cup_{i=1, 2,3}  {\mathcal P}_i ) \le 1 ] \times \frac {C}{|w-z|^{d-2}} \right) $$
which gives a bound of the type $C'N^{6 -2d}$ (the constants $C,C',C'', C_k,\ldots$ in this section do not depend on $N$).
Summing this type of estimate over all points $y, z$ in $\Lambda_N$ leads to
$$ \sum_{y,z \in \Lambda_N}
P [ 0 \leftrightarrow y \leftrightarrow z ] \le C'' N^{6}$$
and more generally, the same argument gives
$$ \sum_{x_1,\ldots ,x_k \in \Lambda_N}
P [ 0 \leftrightarrow x_1 \leftrightarrow x_2 \leftrightarrow \cdots \leftrightarrow x_k  ] \le C_k  N^{-2} N^{4k}.$$
But this now allows to also derive lower bounds for the $n$-point function as well. Indeed, one can now bound for each ${\mathcal P}_3$, the expected number of points on ${\mathcal P}_3$ that $x$ is connected to by using the three-point function bound (this is not specific to this particular percolation model, and has/can be used in other settings where some triangular type conditions hold). This idea then leads to a matching lower bound
 $$ \sum_{x_1,\ldots ,x_k \in \Lambda_N}
P [ 0 \leftrightarrow x_1 \leftrightarrow x_2 \leftrightarrow \cdots \leftrightarrow x_k  ] \ge c_k  N^{-2} N^{4k}, $$
which in turn then allows to obtain some of the results that were expected to hold in these dimensions, for instance using simple estimates in the spirit of \cite {Ai2}.

Again, the statements that follow from such $n$-points estimates are just part of the story since the above switching property based arguments in fact provide more information about the branches appearing in these large percolation clusters. Detailed statements and proofs will appear in \cite {CW,W5}.

\subsection {Intermediate dimensions}

We now make some comments about the simplest case among the intermediate dimensions $d=3,4,5$, namely the case $d=3$ (some of these arguments can be used when $d$ is 4 and 5, but extra work is needed). In the $d=3$ case, one can use the switching property for the exact opposite reason than when $d >6$:  Instead of estimating the probability of the intersection between independent random walks by their mean number of intersection points, we can use the fact that two long random walks have in fact a positive probability of intersecting each other.

Again, we are just going to show one way to get the ball rolling here: Consider four points $x_1, x_2, y_1, y_2$ that are all in $\Lambda_N = [-N,N]^3$ and at distance greater than some constant times $N$ from each other.
Let us first condition on $x_1 \leftrightarrow y_1$. Then, if we apply the switching property , the conditional probability that in the remaining loop-soup in $\Z^3 \setminus \{ x_1, y_1\}$, the two points $x_2$ and $y_2$ are connected is bounded from below by $c / |x_2 -y_2|$. But then, when this happens, the conditional probability that an excursion $e_1$ from $x_1$ to $y_1$ (appearing in the first use of the switching property) and an excursion $e_2$ from $x_2$ to $y_2$ (appearing in the second use of the switching property) do intersect is bounded from below by a constant. We therefore end up with the lower bound
$$ P[ x_1 \leftrightarrow x_2 \leftrightarrow x_3 \leftrightarrow x_4  ] \ge
\frac C {N^2}.$$
The same argument gives a lower bound of the type $C_k' N^{-k}$ for the $2k$-point functions. If one sums over all well-separated $x_1, \ldots, x_{2k}$ in $\Lambda_N$, we end up with
a lower bound
$$ E [ \sum_i |{\mathcal C}_i \cap \Lambda_N|^{2k} ] \ge C_k'' N^{5k},$$
where $({\mathcal C}_i)$ denotes the collection of all loop-soup clusters and $|{\mathcal C}_i|$ their respective Lebesgue measure.

On the other hand, it is also possible to derive matching upper bounds. Let us describe one way to exploit the switching property in that direction.
If we fix $x_2$ and $y_2$, we can estimate the probability $P [ x_1 \leftrightarrow y_1  \leftrightarrow x_2 \leftrightarrow y_2 ]$ by first conditioning on $x_1 \leftrightarrow y_1$ and then on
estimating the probability that (in an independent loop-soup), both $x_2$ and $y_2$ are connected to some of the excursions joining $x_1$ and $y_1$. This means that $x_2$ is connected to one of these excursions and that $y_2$ to one excursion (that could be the same or a different one), but in any case, at most two of the excursions joining $x_1$ to $y_1$ would be involved. We can now notice that the union of two independent excursions can be viewed as a subset of a loop (chosen according to the loop-measure) that is conditioned to go through $x_1$ and $y_1$.
When summing over all $x_1$ and $y_1$ at distance of order $L$ from each other, we end up with the probability that $x_2$ is connected to $y_2$ when one adds one additional loop of size of order $L$ in the loop-soup, and this probability is then comparable to and bounded by $P[x_2 \leftrightarrow y_2]$ itself.
This idea can be used to see that for some $C_k'''$,
$$ E [ \sum_i |{\mathcal C}_i \cap \Lambda_N|^{2k} ] \le C_k''' N^{5k}.$$

The combination of these two matching bounds then yield information about the size of the largest loop-soup clusters in $\Lambda_N$ being $N^{5/2}$. Note that the recent paper \cite {DPR3} obtained precise information about this, namely that the law of this size (renormalized by $N^{5/2}$ is tight (and non-trivial) in the $N \to \infty$ (as well as the corresponding results for $d=4, 5$)  by other means.
Again, this will be further discussed in \cite {W4}.

\subsection {The continuum versions}
\label {Scont}
We conclude with the following point:
In dimensions $d=3,4,5$, one expects (see \cite {W2}) the collection of loop-soup clusters to have a scaling limit when the mesh of the lattice goes to $0$ (just as in $d=2$, see \cite {L2}), that will be coupled with a realization of the Brownian loop-soup in the continuum [and the properties of this scaling limit are then directly related with the large-scale properties of the models on cable-graphs]. In that context, given that the mass of excursions joining two far-away points in the cable-graph goes to $0$ in the scaling limit, one expects that in the continuum, the switching property would hold in a similar way to Theorem~\ref{thm1}: Conditioning two points $x$ and $y$ to be in the same ``continuum sign-cluster'' would amount to just adding one Brownian excursion joining these two points. As we have already briefly mentioned in the introduction, this has actually already be shown to hold in dimension 2 (and a domain with finite Green's function) by J\'ego, Lupu and Qian \cite {JLQ}. We plan to address such ``switching property in the continuum setting results'' in subsequent work \cite {W4}.

\subsection*{Acknowledgments.} This research has been supported by a Research Professorship of the Royal Society. I thank Elie A\"\i d\'ekon, Juhan Aru, Hugo Duminil-Copin, Lorca Heeney, Yves Le Jan, Marcin Lis and Titus Lupu  for their comments on first versions of this paper. Many thanks to Elie A\"\i d\'ekon for pointing out the Pitman-Yor paper \cite {PY}.

\end{document}